\documentclass{imsart}
\RequirePackage{amsthm,amsmath,amsfonts,amssymb}
\RequirePackage[numbers]{natbib}
\usepackage{soul}

\usepackage[utf8]{inputenc}
\usepackage[english]{babel}
\usepackage{bbm}
\usepackage[capitalize,nameinlink]{cleveref}
\usepackage{graphics, graphicx}
\usepackage[ruled,vlined]{algorithm2e}
\usepackage{xcolor}
\usepackage{comment}
\usepackage{relsize}
\usepackage{appendix}

\startlocaldefs
\numberwithin{equation}{section}
\theoremstyle{plain}
\newtheorem{theorem}{Theorem}[section]
\newtheorem{proposition}[theorem]{Proposition}
\newtheorem{definition}[theorem]{Definition}

\theoremstyle{remark}
\newtheorem{example}[theorem]{Example}

\newcommand{\PP}{\mathbb{P}}

\newcommand{\tF}{\tilde{F}}
\newcommand{\ttheta}{\tilde{\theta}}
\newcommand{\tG}{\tilde{G}}
\newcommand{\iidsim}{\overset{i.i.d.}{\sim}}
\newcommand{\indsim}{\overset{indep}{\sim}}

\newcommand{\indep}{\perp \!\!\! \perp}
\newcommand{\Polya}{P\'olya }

\newcommand{\hF}{\hat{F}_n}

\newcommand{\Norm}{\mathcal{N}}

\endlocaldefs

\begin{document}

\begin{frontmatter}
\title{Exchangeability, prediction and predictive modeling in Bayesian statistics}

\runtitle{Prediction and exchangeability}

\begin{aug}
\author[A]{\fnms{Sandra}~\snm{Fortini}\ead[label=e1]{sandra.fortini@unibocconi.it}}
\and
\author[B]{\fnms{Sonia}~\snm{Petrone}\ead[label=e2]{sonia.petrone@unibocconi.it}}

\address[A]{Sandra Fortini is Associate Professor of Statistics at the Department of Decision Sciences, Bocconi University \printead[presep={\ }]{e1}.}
\address[B]{Sonia Petrone is Professor of Statistics, Department of Decision Sciences and Bocconi Institute of Data Science and Analytics, Bocconi University
\printead[presep={\ }]{e2}.}
\end{aug}

\begin{abstract}
There is currently a renewed interest in the Bayesian {\em predictive} approach to statistics. 
This paper offers a review 
on foundational concepts and focuses on \lq predictive modeling\rq, which by directly reasoning on prediction, bypasses inferential models or may characterize them.   
We detail predictive characterizations in exchangeable and partially exchangeable settings, for a large variety of data structures, and hint at new directions.
The underlying concept is that Bayesian predictive rules are probabilistic {\em learning} rules, formalizing through conditional probability how we learn on future events given the available information.
This concept has implications in any statistical problem and in inference, from classic contexts to less explored challenges, such as providing Bayesian uncertainty quantification to predictive algorithms in data science, as we show in the last part of the paper. 
The paper gives a historical overview, but also includes a few new results, presents some recent developments and poses some open questions.

\medskip

\end{abstract}


\end{frontmatter}

\section{Introduction}\label{sec:introduction}
There is currently a renewed 
interest in the {\em Bayesian predictive approach} to statistics.
The approach is just Bayesian, 
but the additional adjective \lq predictive' 
underlines conceptual emphasis on predictive tasks; while the more common 
\lq inferential approach' is centered on inference on parameters, here one focuses on observable quantities and prediction, evaluates models and priors based on their implications on prediction, and even deduce models and parameters from the predictive rule (the long list of references includes \cite{degrootFienberg1982}, \cite{Geisser1993-book}, \cite{goldstein1999}, \cite{clarke2018}). With the major focus on prediction in data science and machine learning (\cite{breiman2001}, \cite{shmueli2010}), this approach appears natural and is adopted in novel  research directions  (\cite{hahnMartinWalker2018}, \cite{fortini2020-newton}, \cite{fong2021}, \cite{holmes2022}
\cite{rigo2023},   \cite{walker2023}).
In fact, the predictive approach has a long tradition in Bayesian statistics and is rooted in its same foundations 
(de Finetti \cite{definetti1931}, \cite{deFinetti1937}, \cite{cifarelli1996}, 
Savage \cite{savage1954}, \cite{dubinsSavage1965},  and 
Diaconis  
\cite{Diaconis1988-recentProgresses}, \cite{diaconisFreedman1980}, 
Regazzini \cite{regazzini-1999-valencia}, \cite{fortini2000}, 
Dawid \cite{dawid1984} and more; see the book by Bernardo and Smith   \cite{bernardoSmith1994}). 

The first aim of this paper is to offer a review, starting from foundations and going through methods for predictive constructions in a variety of contexts, with focus on exchangeable structures, which play a basic role.
Thus we also review, from a  predictive perspective,  the use of exchangeability and of forms of partial exchangeability in Bayesian statistics. 

A second aim of the paper is to show how a Bayesian predictive approach can be usefully adopted in less explored situations, beyond exchangeability; in particular, 
(a) to obtain computationally tractable approximations of (exchangeable) Bayesian inferences and (b) to provide  Bayesian uncertainty quantification of some classes of algorithms (a novel example we provide is online gradient descent),  without the need of an explicit likelihood and prior law. This is developed in the last part of the paper  and relies  on the foundational principles 
that we discuss 
in the first part. 

Along our review, we include a few novel results and  open problems.  
We hope that the paper may be of some interest, especially to young researchers, as both a reminder of the foundations and of some remarkable results,  and as an inspiration for new work.

\subsection{Basic concepts and paper overview}
In Bayesian statistics, prediction is expressed through the {\em predictive distribution} of future observations given the available information. In the simplest setting (and with an abuse of notation, in this introduction identifying distributions through their arguments) one has a sample from a sequence of random variables (r.v.'s) $(X_n)_{n \geq 1}$, 
has specified a conditional model $(X_1, \ldots, X_n) \mid \ttheta\sim$ $p(x_1,\dots,x_n \mid \ttheta)$, $n \geq 1$ 
and a prior distribution $\pi$ on $\ttheta$, and computes the predictive density of $X_{n+1}$ given $x_{1:n} {\equiv} (X_1=x_1, \ldots, X_n=x_n)$  as 
\begin{equation} \label{eq:pred} 
p(x_{n+1} \mid x_{1:n})
= \int_\Theta p(x_{n+1} \mid x_{1:n}, \theta) d\pi(\theta \mid x_{1:n}),  
\end{equation} 
where $\pi(\cdot \mid x_{1:n})$ is the posterior distribution of $\ttheta$ (we use the notation $\ttheta$ to underline that it is a r.v.). Summaries of the predictive distribution naturally include 
point prediction and predictive credible intervals. Thus, while standard frequentist prediction would 
move from a model $(X_1, \ldots, X_n) \sim p_\theta(x_1, \ldots, x_n)$ and  deal with parameters' uncertainty by plugging their estimates into 
$p_{\theta}(x_{n+1}\mid x_{1:n})$,  
in the Bayesian approach uncertainty is taken into account  by  \lq averaging' the possible models  $p(x_{n+1} \mid x_{1:n}, \theta)$ with respect to the posterior distribution of $\ttheta$. 

We already see distinctive features of Bayesian prediction; but this all may sound as
\lq the usual Bayesian story'. Actually, Bayesian statistics is often described as consisting of 
assigning a prior on $\ttheta$ and using  Bayes rule to compute the posterior distribution. Obtaining the predictive distribution as in (\ref{eq:pred}) is then just a matter of computations. 
Of course, Bayesian statistics is deeper than that; and a first basic concept we should recall for this paper is  the interpretation of the Bayesian predictive distribution.
 
Bayesian statistics is about acting under uncertainty, or incomplete {\em information}.  This can be information from the data, from domain knowledge, etc; the point is 
to formalize 
that information, and probability  is the prescribed formal language for this. 
If probability  describes (incomplete) information,
then the evolution of information, or {\em learning}, is expressed through {\em conditional probabilities}. In particular, learning on the next observation based on the observed $x_{1:n}$ is expressed through the conditional distribution 
$p(x_{n+1} \mid x_{1:n})$. This leads us to the interpretation of the Bayesian predictive distribution: it is    a {\em learning rule} that formalizes, through conditional probability, how we learn about future events 
 given the available information. (Thus, it is not meant  
 as the \lq physical mechanism' generating $X_{n+1}$ given the past -- in the classic setting, 
that might be the interpretation of $p_{\theta_0}(x_{n+1}\mid x_{1:n})$ for a true $\theta_0$). 
See e.g. \cite{fortini2016-deFinettiview}.

This principle is the basis of our discussion in the paper, and we return on it in a rather novel way in Section \ref{sec:algorithms}. Here, to see a first implication, let us consider the basic case, random sampling. 
In the Bayesian approach, one does not assume independence, as it would give $p(x_{n+1} \mid x_{1:n})=p(x_{n+1})$, expressing no learning. 
One would rather elicit a  joint probability $p(x_1, \ldots, x_n)$  that expresses dependence: 
not because the $X_i$ are \lq physically' dependent, but because each $X_i$ brings information about the others. The $X_i$ are dependent in our probability assessment formalizing the learning process. 
In random sampling, the natural assessment is that the order 
of the observations does  not bring any information: the $X_i$ are exchangeable.
Then they  are only {\em conditionally} independent. 
We devote substantial space in the paper to exchangeability; simply because 
it is the natural predictive requirement 
in random sampling, and random sampling is the basic setting.
The fundamental  concepts 
are treated in \cref{sec:exch}.

In practice, we usually specify the joint distribution $p(x_1, \ldots, x_n)$, for any $n$, with the help of models and parameters 
\begin{equation} \label{eq:inf-model}
p(x_1, \ldots, x_n)=\int p(x_1, \ldots, x_n \mid \theta) d \pi(\theta);
\end{equation}
and compute the predictive distribution as in (\ref{eq:pred}). But, especially if interest is in prediction, we could in principle bypass the inferential model and directly specify the predictive distributions - typically, the one-step-ahead predictions,  which give, for any $n$, \begin{equation} \label{eq:pred-model}
p(x_1, \ldots, x_n)=p(x_1) p(x_2\mid x_1) \cdots p(x_n \mid x_{1:n-1}).
\end{equation}
In this {\em predictive approach}, that we refer to as \lq \lq predictive modeling", 
one reasons on the observable quantities, for example on symmetry properties as in the case of exchangeability, and on what is the relevant information in the sample for prediction, or desirable properties of the predictive learning rule.  This is well rooted in  Bayesian foundations and is particularly attractive in complex settings where models and parameters tend to lose interpretability. 
Still, predictive modeling  may seem quite  impracticable; it has in fact a long tradition, 
however the available literature is rather fragmented. Thus in Section \ref{sec:methodsConstruction} our effort is  to trace 
concepts and methods that may provide a methodological basis to predictive constructions. 
We mostly refer to exchangeable settings, but   
a predictive approach can be taken for any kind of data structures;  
see \cite{rigo2023}.

{\em Prediction and inference.} 
Predictive modeling is also intriguing as a form of \lq \lq Bayesian learning without the prior". In fact, an inferential model and a prior law may be implicitly subintended, and unveiling them is important both practically and conceptually.
This is typically obtained through representation theorems; roughly speaking, one can move from the predictive specification (\ref{eq:pred-model}) of the joint distribution $p(x_1, \ldots, x_n)$, for any $n \geq 1$,  and might {\em represent} it in a form as  (\ref{eq:inf-model}); see Section \ref{sect:exch-caratterization}.  
Although an inferential model is not
needed  in a purely predictive approach, representation theorems  
significantly provide the link from prediction to inference. 
de Finetti's representation theorem 
has a central role in Bayesian statistics as it leads 
from foundations, where probability is expressed on {\em observable} events (see Section 2), to inference. 
de Finetti moves from exchangeability of the observable $X_i$, and the representation theorem  
gives the theoretical justification of the basic Bayesian inferential scheme where the parameter $\ttheta$ is random and the $X_i$ are conditionally i.i.d. given $\ttheta$, as an implication of exchangeability.
Moreover, it shows how the inferential model is related to frequencies. In \cref{sec:prediction-frequency}, we will underline how {\em prediction}  is related to frequencies, thus to the inferential model, and in particular we give a result (\cref{sec:predictive-based appr}) showing how the uncertainty expressed in the posterior distribution is determined by the way the predictive distribution learns from the data. 

Representation theorems have been extended in numerous directions (Sect 2.3 of \cite{cifarelli1996} includes extensive references) and predictive constructions are applied well  beyond simple random sampling.  
In \cref{sec:partial exchangeability}  we consider more structured data for which it is natural to express a predictive judgment of  
{\em partial} exchangeability;  we provide predictive characterizations of some 
forms of partial exchangeability (Theorems \ref{th:exch-by-pred}, \ref{th:partial-by-pred} and \ref{th:markov-by-pred}), and review de Finetti-like representation theorems, which give the predictive-theoretical basis in many problems including stochastic design regression (as reducible to random sampling), fixed design regression and  multiple experiments (\cref{sec:partial}), Markov chains (potentially, models for temporal data based on  Markov chains) (\cref{sec:markov}) and arrays and networks data (\cref{sec:RCE}). 
There are authoritative and comprehensive references on the theory of exchangeability, see Kingman \cite{Kingman78exchangeability}, Aldous \cite{aldous1985}, 
Kallenberg \cite{kallenberg2005}, to which we refer interested readers. The more specific aim of our - necessarily brief - review is to point out  some main aspects that we believe are relevant in Bayesian statistics from a predictive perspective. 

{\em Open directions}. Although the above discussion shows that the predictive approach is theoretically sound and that predictive modeling can be applied in many contexts,
we acknowledge that proceeding solely through predictive constructions  may not be easy, especially if one wants to satisfy exchangeability constraints. On the other hand -- and this is a further point we want to make in this paper -- there are many predictive algorithms in data science that lack clean uncertainty quantification, or there are, in fields such as economics, subjective predictions implicitly guided by the agent's explanation of the phenomena, that would be interesting to reveal 
(see e.g. \cite{AugenblickRabin2021}). A Bayesian predictive approach can be usefully employed. In particular, we show 
that some classes of recursive predictive algorithms can in fact be read 
as Bayesian predictive learning rules, 
that assume exchangeability only asymptotically. The relevance of this approach is not merely theoretical, but allows to understand the underlying modeling assumptions and to provide formal  uncertainty quantification, and can lead to principled extensions. This is treated in Section \ref{sec:algorithms}. 

Brief final remarks conclude the paper.  All the proofs are collected in the Supplement \citep{supplement}.

\subsection{Preliminaries and notation}
In this paper, all the random variables take values in a Polish space $\mathbb X$, endowed with its Borel sigma-algebra $\mathcal X$. 
The topology on spaces of probability measures is implicitly assumed as the topology of weak convergence. Hence, for any $P_n$ and $P$,  $P_n\rightarrow P$ means weak convergence. 

The underlying probability space $(\Omega,\mathcal F,\PP)$ for a random sequence $(X_n)_{n\geq 1}$ is implicitly assumed to be the canonical space $(\mathbb X^\infty,\mathcal X^\infty, \PP)$, where $\PP$ is the probability law 
of the sequence, denoted as $(X_n)_{n\geq 1}\sim \PP$. We write $\PP$-a.s. for \lq\lq with $\PP$-probability one''. We use the short notation $x_{1:n}$ for $(X_1=x_1,\dots,X_n=x_n)$.
All conditional distributions must be understood as regular versions. 
For random variables taking values in Euclidean spaces, we denote with the same symbol a probability measure and the corresponding distribution function. 
Sequences are denoted as $(Z_n)_{n \geq 1}$ and arrays as $[Z_{i,j}]_{i \in I, j \in J}$. 

\section{Exchangeability and prediction}\label{sec:exch}

We begin by recalling the foundational role of prediction in Bayesian statistics and the  notion of exchangeability as a basic predictive judgment. 

\medskip

Bayesian statistics has decision-theoretic  roots in the work of the 1920s
in mathematical logic aimed at founding a normative theory of rational decisions under risk (Ramsey \citep{ramsey1926}, and later, Savage \cite{savage1954}, \cite{dubinsSavage1965}; two book references are  \cite{bernardoSmith1994} and \cite{parmigianiInue2029}). In this perspective, probability arises  as the prescribed rational ({\em coherent}; see \cite{cifarelli1996}) formalization of the agent's information on uncertain events, as advocated in the foundations of modern Bayesian statistics by 
Bruno de Finetti; see e.g. \cite{deFinetti1937} and 
\citep{deFinetti1970}. 
de Finetti emphasises  that probability is 
expressed on {\em observable} events 
(we do not discuss, here, issues on the notions of observability or of imprecise probability; see e.g. \cite{williams1976}). In this perspective, unobservable parameters are not assigned 
a probability {\em per se}, but simply as an intermediate step for ultimately expressing the 
 probability of observable events. They are just 
a tool in the learning process that goes from past observable events to prediction of future events. 
Of course, parameters may be interpretable if not strictly observable, and inference is a core problem; but it is prediction that has a foundational role. 

The focus on probability of observable events 
is well demonstrated in de Finetti's notion and use of exchangeability. As mentioned in the Introduction, in the context of homogeneous replicates of an experiment 
(random sampling) the researcher would judge that the labels of the 
$X_i$ 
\lq \lq do not matter". This is formalized through a joint probability law that is invariant under permutations of the labels:
$$ (X_1, \ldots, X_n) \overset{d}{=} (X_{\sigma(1)},\ldots, X_{\sigma(n)}) $$
for each permutation $\sigma$ of $(1, \ldots, n)$, where $\overset{d}{=}$ means equal in distribution.  
An infinite sequence $(X_n)_{n \geq 1}$ 
is exchangeable if it is invariant to each finite permutation of $\{1, 2, \ldots\}$, i.e. each permutation that only switches a finite set of indexes. 
Exchangeability is an elegant 
probabilistic structure 
and  exchangeable processes arise in many fields.
In de Finetti's work on Bayesian foundations, 
however, exchangeability is not meant as a physical property of the sequence $(X_n)_{n\geq 1}$, but as an expression of the agent's information. 

\begin{example} 
Consider random sampling from a two-color urn, and let $X_i=1$ if the color of the ball picked on the $i$-th draw is white, and zero otherwise. The agent judges that the order of the draws is not informative and the sequence $(X_n)_{n\geq 1}$ is exchangeable. By the representation theorem (\cref{sect:exch-caratterization}), $(X_n)_{n\geq 1}$ has the same probability law of a sequence 
arising from an experiment where the urn composition 
is picked from a \lq prior' distribution and balls 
are then sampled at random with replacement. 
The physical experiment is not as such: the urn composition is not sampled, it 
is given although unknown.  Here, exchangeability is not referring to the mechanism generating the data, but to the way we use information.
$\square$\end{example}
We should keep in mind this use of exchangeability in what follows. See also \cite{fortini2023}, and the discussion in \cite{vonPlato1982} for the more general setting of stationary sequences.

\medskip
Although exchangeability is a predictive requirement, it has an immediate inferential implication, 
established by the celebrated de Finetti's representation theorem. 

\begin{theorem}[Law of large numbers and representation theorem for infinite exchangeable sequences]\label{th:representationExch}
Let $(X_n)_{n\geq 1}$ be an infinite exchangeable sequence and denote by $\PP$ its probability law. Then:
    \begin{itemize}
\item[{\rm i)}]
With $\PP$-probability one, the sequence of the  empirical distributions 
$\hat F_n=\frac 1 n \sum_{i=1}^n\delta_{X_i}$
converges weakly as $n\rightarrow\infty$ to a random distribution $\tF$: 
$$ \hat{F}_n \rightarrow \tF ; $$ 
\item[{\rm ii)}] 
For all $n \geq 1$ and measurable sets $A_1, \ldots, A_n$, 
\begin{equation} \label{eq:inferential} 
\PP(\!X_1\!\! \in\! \!A_1, \ldots,\! X_n\!\! \in\! \!A_n\!)\!=\! \!\int \! \prod_{i=1}^n\! F(A_i) d\pi\! (\!F),
\end{equation}
where $\pi$ is the probability law of $\tF$.
\end{itemize}
\label{th:representationTh-exch}
\end{theorem}

See Aldous \cite{aldous1985}, who refers to $\tF$ as the {\em directing random measure} of the exchangeable sequence $(X_n)_{n\geq 1}$. The re\-pre\-sentation {\rm ii)} is often phrased as 
\lq \lq $X_i\mid\tilde F=F\stackrel{i.i.d.}{\sim} F$, with $\tF \sim \pi$\rq \rq; a subtle difference is that this latter formulation may (in principle, misleadingly) suggest the existence of a true $F$. In Bayesian inference, $\tF$ plays the role of the statistical model, and its probability law is  the prior. 
The prior law is unique, and is a probability measure on the class of all the possible distributions on the sample space. The representation theorem is a high-level result: the probability law $\PP$ characterizes the random $\tF$; in other words, it shapes it (the model) through its implied distribution (the prior). In applications, one has to choose a specific law $\PP$. In particular, further 
information may restrict the support of the prior to a parametric class, so that  $X_i \mid \ttheta \iidsim p(\cdot \mid \ttheta)$ (see Section \ref{subsec:sufficiency}). In this paper we will mostly
keep the general  
framework (\ref{eq:inferential}).

\medskip

\noindent {\em Remark}. 
Note that $\tF$ in \cref{th:representationExch} 
is random; as the limit of the empirical distributions, it  depends on $(X_1, X_2, \ldots)$. Given a sample path $\omega=(x_1, x_2, \ldots)$, we have a realization of the random $\tF$, that we denote by $\tF(\cdot)(\omega)$.
In fact, for i.i.d. observations from a
distribution $F$, the limit of the empirical distribution is $F$;  
the fact that the limit is instead random for exchangeable sequences may sound weird. Formally, this is because exchangeable sequences are {\em mixtures of i.i.d. sequences}; let us
give some intuition. 
By the representation theorem, an exchangeable sequence $(X_n)_{n\geq 1}$
can be obtained by first picking a distribution $F$ from the prior law, then sampling the $X_i$ at random from $F$. 
Having picked $F$, if we restrict ourselves to the set of the sample paths $\omega=(x_1, x_2, ...)$ that we may obtain by sampling at random from it, 
we have the usual properties of the i.i.d. case; in particular, for  
almost all these $\omega$ 
the empirical distribution converges to $F$; thus $\tF(\omega)=F$, which is not random.  
However, when we observe a finite sample $(x_1, \ldots, x_n)$, we do not know what $F$ was
chosen, hence the limit of the empirical distribution may still be any distribution we could have picked from the prior. We would know which one 
if we could observe  the entire $\omega=(x_1, x_2, \ldots)$  
and thus see the limit of the empirical distribution, that is the realization $\tF(\cdot)(\omega)$ of the random $\tF$.

\subsection{Predictive characterization of exchangeability} \label{sect:exch-caratterization}

The representation theorem allows us to specify an exchangeable probability law through the usual inferential scheme. In a predictive approach, however, we would avoid models and priors and directly specify it through the predictive rule. This is the core of predictive modeling, beyond exchangeability.

For any probability law $\PP$ for the sequence $(X_n)_{n\geq 1}$, define the {\em predictive rule} 
as the sequence of predictive distributions  
$P_0(\cdot) \equiv \PP(X_1 \in \cdot)$ and, for $n \geq 1$,
\begin{equation} \label{eq:predPn}
P_n(\cdot) 
\equiv \PP(X_{n+1}\in \cdot \mid X_1,\dots,X_n),
\end{equation}
and let us denote by  $P_n(\cdot\mid x_{1:n})$ its realization for $x_{1:n}$. In particular, if  $\PP$ is exchangeable, the predictive rule 
is obtained as
$P_0(\cdot) 
= E(\tF(\cdot))$ and 
$P_n(\cdot)=E(\tF(\cdot)|X_1,\dots,X_n)$ for $n \geq 1$. 

In predictive modeling, one moves from the predictive rule to specify the probability law of the process $(X_n)_{n \geq 1}$. More formally,
one can  assign a sequence $(P_n)_{n\geq 0}$ of probability kernels
(or a {\em strategy}, \cite{dubinsSavage1965}, \cite{rigo2023}). Then by
 Ionescu-Tulcea theorem (see \cite{kallenberg2002} Theorem 5.17 and Corollary 5.18) 
there exists a unique probability law $\PP$ for a process $(X_n)_{n\geq 1}$ such that $X_1 \sim P_0$ and for every $n \geq 1$, $P_n$ is the conditional distribution of $X_{n+1}$ given $(X_1, \ldots, X_n)$.
(Given this equivalence, we will use the notation $(P_n)_{n\geq 0}$ 
to represent both the sequence of  the predictive distributions from 
 a given $\PP$, and a strategy).

Thus, the probability law $\PP$ of a process $(X_n)_{n \geq 1}$ is uniquely defined (characterized) by the sequence of predictive distributions $(P_n)_{n \geq 0}$.
A natural question is under what conditions on the $P_n$ one obtains an {\em exchangeable} law $\PP$. 
This problem has been addressed in \cite{fortini2000}.
\begin{theorem}[\cite{fortini2000}, Proposition 3.2 and Theorem 3.1]\label{th:exch-by-pred}
 Let  $(X_n)_{n\geq 1}\sim\PP$ be an infinite sequence of r.v.'s,  with predictive rule $(P_n)_{n \geq 0}$. 
  Then $(X_n)_{n\geq 1}$
 is exchangeable if and only if, for every $n\geq 0$, the following conditions hold:
\begin{itemize}
\item[{\rm i)}] For every 
$A$, $P_n(A\mid x_{1:n})$ is a symmetric function of $x_1,\dots,x_n$;
\item[{\rm ii)}] 
The set function 

$ (A,B) \rightarrow  \int_A P_{n+1}(B\mid x_{1:n+1})dP_n(x_{n+1}\mid x_{1:n})$ 

is symmetric in $A$ and $B$,
\end{itemize}
where $P_0(\cdot\mid x_{1:0})$ is meant as $P_0(\cdot)$. 
\end{theorem}

Condition i) requires that, for every $n\geq 1$, the predictive distribution of $X_{n+1}$ is a function of the empirical distribution of $(x_1,\dots,x_n)$; which is a necessary condition for exchangeability.  
As well,  given  $x_{1:n}$,
the predictive distribution of $(X_{n+1}, \ldots, X_{n+k})$ should be invariant under permutations  of the $k$ future observations, since under exchangeability the joint distribution of $(X_1, \ldots, X_{n+k})$ is symmetric. 
Condition {\rm ii)} only asks that the next $k=2$ observations can be permuted. 

\subsection{Prediction, frequency, models} \label{sec:prediction-frequency}

Although there are no formal constraints in assigning a predictive rule $(P_n)_{n \geq 0}$, we  aim for our predictions to be consistent with facts. 
For exchangeable sequences, the following property relates prediction to frequency.
\begin{proposition} \label{prop:predconv}
Let $(X_n)_{n \geq 1} \sim \PP\,$ 
be an exchangeable sequence, with predictive rule  $(P_n)_{n \geq 0}$.  
Then,  
with probability one, for $n \rightarrow \infty$ the sequence of predictive distributions converges, 
and its limit coincides with the limit of the empirical distributions: 
\begin{equation} \label{eq:predConvProp} 
P_n \rightarrow \tF ,   \quad  \mbox{$\PP$-a.s.},
\end{equation}
with $\tF$ as in \cref{th:representationTh-exch}.
\end{proposition}
A proof  
is given in  \cite{aldous1985}, Lemma 8.2 page 61. 
In fact, the result 
remains valid under the less restrictive condition that  $(P_n(A))_{n\geq 0}$ is a martingale for every $A$, without the need for $(X_n)_{n\geq 1}$ to be exchangeable \citep{berti2004}. 
We return on this point in more details in \cref{sec:asymptotic}.  

\begin{example}\label{example:multinomial}
Consider an exchangeable sequence $(X_n)_{n\geq 1}$ with $X_i \in \{1,\ldots, k\}$. 
Then the empirical distribution is characterized by the vector of relative frequencies $n_j/n$,
and any predictive distribution must be a function of $(n_1, \ldots, n_k)$, i.e.  
$p_n(j) \equiv \PP(X_{n+1}=j \mid x_{1:n})= \PP(X_{n+1}=j \mid n_1, \ldots, n_k), j=1, \ldots, k$. 
For any $j$, with probability one the relative frequency $n_j/n$
and the predictive probability $p_n(j)$ converge 
to the same random limit $\tilde{p}_j$.  
The statistical model is a discrete distribution on $\{1, \ldots, k\}$ with masses $(\tilde{p}_1, \ldots, \tilde{p}_k)$ and the prior is the probability law of the random limit $(\tilde{p}_1, \ldots, \tilde{p}_k)$. 
$\square$\end{example}

In Bayesian statistics, Proposition \ref{prop:predconv} ensures that, with probability one, our predictions will  adjust 
to frequencies; in other words, the predictive distribution $P_n$ and the empirical distribution $\hF$ will be close. Several refinements of this property are available, 
as well as quantitative bounds (\citep{Dolera2019}, and references therein; see also \cite{diaconisFreedman1990}). 

Proposition \ref{prop:predconv} also shows that, for exchangeable sequences, the statistical model is  the limit of the predictive distribution; which is also 
the limit of the empirical distribution. Hence,  the uncertainty on the model at a finite  $n$
is  uncertainty on their common limit.
It is this  uncertainty that is expressed in the posterior distribution of $\tF$, as we will illustrate in  \cref{sec:predictive-based appr}, expanding from \cite{fortini2020-newton}. This %
is also the basic principle that underlines the interpretation of uncertainty in terms of \lq \lq missing observations" in \cite{fong2021}. 

\vspace{2mm}
{\em Remark}. 
de Finetti  proved the convergence property (\ref{eq:predConvProp})  
for exchangeable binary sequences $(X_n)_{n\geq 1}$, and it is interesting to note that he used this result to give an explanation in terms of prediction of  
the frequentist viewpoint on probability \citep{deFinetti1937}. He considers replicates of an experiment with binary outcome where a frequentist researcher assumes 
that $\PP(X_{n+k}=1 \mid x_{1:n})=p$ for any $k \geq 1$ and,  
for $n$ large, 
estimates $p$ with the relative frequency $\hat{p}_n=\sum_{i=1}^n x_i /n$. 
Exchangeability makes the frequentist prediction, namely $\PP(X_{n+k}=1 \mid x_{1:n}) \simeq \hat{p}_n$, 
\lq \lq permissible", by the result (\ref{eq:predConvProp}).  
See \cite{cifarelli1996}, Sect 2.3. 

\subsection{Asymptotic exchangeability.}\label{sec:asymptotic}

A natural question is whether there is a reverse implication of Proposition \ref{prop:predconv}. Exchangeability of $(X_n)_{n\geq 1}$ implies that $P_n \rightarrow \tF$, $\PP$--a.s., but convergence of $(P_n)_{n\geq 0}$ to a random probability measure does not imply exchangeability. However, it does so asymptotically. 
A sequence $(X_n)_{n\geq 1}$ is asymptotically exchangeable with limit directing random measure $\tF$ (shortly,  {\em $\tF$-asymptotically exchangeable})   
if, for $n \rightarrow \infty$  
$$ (X_{n+1}, X_{n+2}, \ldots ) \overset{d}{\rightarrow} (Z_1, Z_2, \ldots) , $$
where the sequence $(Z_n)_{n\geq 1}$ is exchangeable and has directing random measure $\tF$. 
It can be proved that, if the sequence of predictive distributions $(P_n)_{n\geq 0}$ of $(X_n)_{n\geq 1}$ converges to a random probability measure $\tF$, then $(X_n)_{n\geq 1}$ is $\tF$-asymptotically exchangeable (\cite{aldous1985}, Lemma 8.2). 
Roughly speaking, for $n$ large
$$X_n \mid \tF \overset{iid}{\approx} \tF,
$$
where $\tF$ has a prior law induced by the predictive rule. 

Interestingly, convergence of the sequence $(P_n)_{n \geq 0}$ to a random probability measure, thus asymptotic exchangeability, holds if $(P_n)_{n \geq 0}$ is a martingale, or, equivalently, if 
the sequence $(X_n)_{n \geq 1}$ is  {\em conditionally identically distributed} (c.i.d.); that is, if it satisfies
\begin{equation} \label{eq:mgspreadability}
(X_1, \ldots, X_n, X_{n+1}) \overset{d}{=} (X_1, \ldots, X_n, X_{n+k}), 
\end{equation}
for all integers $k \geq 1$ and $n \geq 1$; i.e., 
conditionally on the past, all future observations are identically distributed. The property (\ref{eq:mgspreadability}) was considered by Kallenberg as a weak invariance condition that, for stationary sequences, is equivalent to exchangeability (\cite{kallenberg1988}, Proposition 2.1). 
He also noted that (\ref{eq:mgspreadability}) is equivalent to  $(X_1, \ldots, X_n, X_{n+1}) \overset{d}{=} (X_1, \ldots, X_n, X_{n+2})$ for all $ n \geq 1$.
The term c.i.d. was introduced by Berti {\em et al.} \cite{berti2004}, who proved, among other results, 
that the c.i.d. property is equivalent to $(P_n)_{n \geq 0}$ being a measure-valued martingale with respect to the natural filtration
of $(X_n)_{n \geq 1}$. The martingale property means that the sequence of random measures $(P_n)_{n\geq 0}$ satisfies
\begin{align*}
E(P_{n+1}(A)\mid X_1,\dots,X_n)=P_n(A)
\end{align*}
for every $n\geq 0$ and every measurable set $A$.  See \cite{horowitz1985}. 

For exchangeable sequences, it is straightforward to show that the predictive rule is a martingale. But the martingale condition is weaker than exchangeability; still, 
remarkably, it is sufficient to prove the convergence result in Proposition \ref{prop:predconv}: for a c.i.d. process $(X_n)_{n \geq 1}$, the sequence of the empirical distributions converges $\PP$-a.s. to a random probability distribution $\tF$, and the sequence of predictive distributions not only converges (being a bounded martingale), but converges to the same limit $\tF$ (\cite{berti2004}, Theorem 2.5).

Thus, a c.i.d. sequence is asymptotically exchangeable. 
However, it is not generally exchangeable. The property that is broken is stationarity: the researcher is assuming a temporal evolution in the process. It is however a specific form of evolution: marginally, the r.v.'s 
are identically distributed, and the process converges to a stationary - thus, together with the c.i.d. property, exchangeable - state (see also \cite{fortini2018} and \cite{fortini2020-newton}). For more developments, we refer to \cite{rigo2023}.  We 
return to asymptotic exchangeability and c.i.d. sequences in Section \ref{sec:algorithms}.

\subsection{Predictive-based approximations of the posterior distribution.}\label{sec:predictive-based appr}
In the usual inferential setting, one computes the posterior distribution and obtains the
predictive distribution as in expression 
(\ref{eq:pred}).  In predictive modeling, the order is reversed; here,  from the predictive assumption of   exchangeability of the $X_i$, we have obtained the implied inferential scheme.
Can we also revert the order in expression (\ref{eq:pred}), i.e. go from the predictive rule to the posterior distribution, and what would be the implications on inference? 
In this section we address this question and show two implications; namely, two predictive-based approximations of the posterior distribution. 

For brevity, here we consider $X_i \in \mathbb{R}$. 
In the exchangeable setting, with no parametric restrictions, we have  
$X_i \mid \tF \iidsim \tilde F$ and we are used to think of 
the prior and posterior distributions on $\tF$ as expressing uncertainty on the true distribution, say $F_0$. 
In fact, as briefly anticipated in \cref{sec:prediction-frequency}, what the prior and the posterior distributions are expressing is the uncertainty about the common limit $\tF$ of the empirical and the predictive distributions. If we knew the entire sample path $\omega=(x_1, x_2, \ldots)$, we would know the limit, namely $\tF(\cdot)(\omega)$,  
and there would be no uncertainty left.
Given a finite sample $(x_1,\dots,x_n)$, we are still uncertain about the limit, and this uncertainty is expressed through the posterior distribution. 
The following approximations of the posterior distribution are based on this principle. 

\medskip
\noindent {\em A predictive-based simulation scheme.} 

First, leveraging on 
Proposition \ref{prop:predconv},  we can provide 
a predictive-based sampling scheme (\cite{fortini2020-newton}, \cite{fortini2023}) to approximate the prior and the posterior distribution of $\tF$; in practice we use $[\tilde F(t_1),...,\tilde F(t_k)]$  for $t_1,...,t_k$ in a grid of values.
Assume that $P_0(\cdot)=E(\tF(\cdot))$ is continuous in $t_1, \ldots, t_k$, which implies that, $\PP$-a.s.,  $\tF$ is  continuous at those points so that $\lim_n \hat{F_n}(t_j) = \lim_n P_n(t_j) = \tF(t_j)$ for any $t_j$, $\PP$-a.s.

In principle, given the predictive rule, one can generate $\omega=(x_1, x_2, \ldots)$ by sampling $x_1$ from $P_0$, then $x_2$ from 
$P_1(\cdot \mid x_1)$ 
and so on; and,  having $\omega=(x_1, x_2, \ldots)$, can obtain $\tF(t_j)(\omega)$ for $j=1, \ldots, k$;  which is 
a sample  from the prior law of $[\tF(t_1), \ldots, \tF(t_k)]$. Repeating $M$ times gives a Monte Carlo sample of size $M$ from the prior.

To simulate from the posterior law given $(x_1,\dots,x_n)$, one can proceed similarly by generating the missing observations $(x_{n+1}, x_{n+2}, \ldots)$ from the predictive rule to complete $\omega$, and repeat $M$ times to obtain a sample of size $M$ from the posterior. Of course, in practice, one would truncate $\omega$ to a finite sequence 
$(x_1, \ldots, x_N)$ 
with $N$ large, and approximate $\tF(t)(\omega)$ with $P_N(t\mid x_{1:N})$, 
or with $\hat{F}_N(t)$, for each $t$ in the grid. 

A similar  
predictive principle is considered in the 
interesting developments by Fong, Holmes and Walker \citep{fong2021} of the Bayesian bootstrap in a parametric setting: samples from a martingale posterior distribution are obtained by Doob theorem \citep{doob1949},  after simulating  future observations from a sequence of predictive distributions (see also \cite{holmesWalker2023}). 

\medskip
\noindent {\em A predictive-based asymptotic approximation of the posterior distribution}. 

One can also obtain a predictive-based analytic approximation of the posterior distribution for large $n$. By Proposition \ref{prop:predconv}, $|\tF(t)- P_n(t)| \rightarrow 0$, $\PP$-a.s., 
for any continuity point $t$ of $\tF$, 
and because $(P_n(t))_{n \geq 0}$ is a martingale, one could use
martingale central limit theorems
to give asymptotic approximations of $(\tF(t)-P_n(t))$; yet, this  
would not inform on its behavior {\em conditionally} on the data. 
The following result uses a central limit theorem for martingales in terms of {\em a.s. convergence of conditional distributions} \citep{crimaldi2009}. This type of convergence has been applied in probability for other aims
and was used in a novel way in Bayesian statistics by \cite{fortini2020-newton} to inform on the asymptotic 
form of the {\em posterior} distribution.  
Here, we provide an asymptotic Gaussian approximation of the joint posterior distribution of $[\tF(t_1), \ldots, \tF(t_k)]$, extending a result in \cite{fortini2023}. Because $E(\tF(t) \mid X_1,\dots X_n)
=P_n(t)$, 
the approximation is centered on $P_n$. 
Then we look at {\em how} the predictive rule learns from the data, introducing
the notion of {\em predictive updates}.
As a fresh observation becomes available, the predictive distribution is updated by  incorporating the latest information, and for $n \geq 1$ and $t \in \mathbb R$ we denote by 
$$\Delta_{t,n}=P_n(t)-P_{n-1}(t)$$ 
the $n$-th update of the predictive distribution function at the point $t$ as $X_n$ becomes available.
For a given $t$, the predictive updates $\Delta_{t,n}$ eventually converge to zero, since $P_n \rightarrow \tF$, and the rate of convergence is generally of the order $1/n$, as discussed in \cite{fortini2023}.  The following proposition shows that the convergence of $ \sqrt{n}(\tilde{F}(t) - P_n(t)) $ depends on the asymptotic behaviour of $ (n\Delta_{t,n})_{n\geq 1} $.
For a grid of points  ${\bf t}=(t_1,\dots,t_k) \in 
\mathbb R^k$,   we denote by  
${\bf \Delta}_{{\bf t},n}=[\Delta_{t_1,n}\;\dots \Delta_{t_k,n}]^T$ the column vector of the updates of $(P_n(t_1),\dots,P_n(t_k))$.
The proposition below holds for exchangeable sequences, but more generally for sequences whose predictive rule is a martingale, i.e. for c.i.d. sequences.
\begin{proposition} \label{prop:credible}
Let $(X_n)_{n \geq 1} \sim \PP$ 
be a c.i.d. 
sequence of real-valued r.v.'s, with predictive rule $(P_n)_{n \geq 0}$, and take  ${\bf t}=(t_1,\dots,t_k)$ such that $\PP(X_1\in \{t_1,\dots,t_k\})=0$. 
Suppose that the predictive updates satisfy
\begin{align*}
    &E(\sup_n \sqrt n |\Delta_{t_i,n}|)<+\infty& (i=1,\dots,k),\\
    &\sum_{n=1}^\infty n^2 E(\Delta_{t_i,n}^4)<+\infty& (i=1,\dots,k),\\
    &E(n^2 {\bf \Delta}_{{\bf t},n}{\bf \Delta}_{{\bf t},n}^T \mid X_1,\dots,X_{n-1}) 
    \rightarrow {U}_{\bf t} & \PP\mbox{-a.s.} , 
    \end{align*}
for a positive definite random matrix  ${U}_{\bf t}$.
Define, for every $n\geq 1$,
 \begin{equation} \label{eq:Vn}
       V_{n, {\bf t} }
    = \frac 1 n \sum_{m=1}^nm^2 {\bf \Delta}_{{\bf t},m}{\bf \Delta}_{{\bf t},m}^T.
\end{equation}
 Then,   $\PP$-a.s.,  
$V_{n, {\bf t} }$ converges to $U_{\bf t}$ and
     \begin{eqnarray*}
& \sqrt n \;V_{n,{\bf t}}^{-1/2}\left[\begin{array}{c}\tilde F(t_1)-P_n(t_1)\\ \dots \\\tilde F(t_k)-P_n(t_k)\end{array}\right] \mid X_1,\dots,X_n 
& \overset{d}{\rightarrow} \Norm_k(0, I) 
\end{eqnarray*}
as $n \rightarrow \infty$, where $\Norm_k(0,I)$ denotes the  $k$-dimensional standard Gaussian distribution. 
\end{proposition}
Informally, for $n$ large, 
$$ \left[\begin{array}{c} \tF(t_1)\\ 
\vdots \\
\tF(t_k)) \end{array}\right]
\mid x_{1:n}  \approx \Norm_k \left( 
\left[\begin{array}{c} P_n(t_1)\\ \vdots \\P_n(t_k)) \end{array}\right],
\frac{V_{n, {\bf t}}}{n}\right)  
$$
for $\PP$-almost all sample paths $\omega=(x_1, x_2, \ldots)$.\\
Proposition \ref{prop:credible}  allows to compute asymptotic credible sets. For example,  a $(1-\alpha)$ marginal asymptotic credible interval for $\tF(t)$ given $x_{1:n}$ is 
$$
\left[ P_n(t ) - z_{1-\alpha/2} \sqrt{\frac{V_{n,t}}{n}}, P_n(t ) + z_{1-\alpha/2} \sqrt{\frac{V_{n,t}}{n}} \right]
$$
with $z_{1-\alpha/2}$ denoting the $1-\alpha/2$ quantile of the standard normal distribution and $V_{n,t}=\frac 1 n \sum_{m=1}^n m^2\Delta_{t,m}^2$.

The proof of Proposition \ref{prop:credible} is given in Section A2 
of the Supplement \citep{supplement}, and consists of two steps. First we prove (Proposition A2.1 
) that, under the given conditions on the predictive updates, 
\begin{equation} \label{eq:prova2}
 \sqrt n \left[\begin{array}{c}\tilde F(t_1)-P_n(t_1) 
 \\ \dots \\\tilde F(t_k)-P_n(t_k)
 \end{array}\right] \mid x_{1:n} 
 \overset{d}{\rightarrow} \Norm_k(0,  U_{\bf t}(\omega)) 
\end{equation}
for $\PP$-almost all $\omega=(x_1,x_2,\dots)$.  
Then we prove that the asymptotic result remains valid if the matrix $ U_{\bf t}$, that depends on the whole sequence $(X_1,X_2,\dots)$,  is replaced by 
its \lq\lq estimate'' $V_{n,{\bf t}}$, that only depends on $(X_1,\dots,X_n)$.

Proposition \ref{prop:credible} gives  sufficient conditions that could possibly be relaxed; also, other choices of $V_{n,t}$ can be envisaged. 
Note that the result  is given under the  law $\PP$, thus, although having a similar flavor, it differs from Bernstein-von Mises asymptotic Gaussian approximations, which are stated with respect to a law $P_{F_0}^\infty$ that assumes that the $X_i$ are i.i.d. from a true distribution $F_0$. Moreover, here the asymptotic variance is expressed in terms of the predictive updates. 

As \lq \lq $\PP$-probability one" results, our findings 
may rather be regarded as a refinement of Doob's theorem for inverse probabilities in the nonparametric case; see point (ii) in Section 4 of Doob \cite{doob1949} 
(for us limited to the finite-dimensional distributions). 
For an exchangeable law $\mathbb P$, 
Doob's theorem ensures that, with $\PP$-probability one i.e. for $\PP$-almost all $\omega=(x_1, x_2, \ldots)$, the posterior expectation $E(\tF(\cdot)\mid x_{1:n})$ converges to
$F=\tF(\cdot)(\omega)$ and
the posterior variance goes to zero, so that the posterior distribution of $\tF$ concentrates around $F$. Proposition \ref{prop:credible} further describes how the posterior distribution of $\tilde F$ concentrates around its conditional expectation:   the asymptotic distribution is Gaussian, and in particular, the rescaled asymptotic variance depends on  how the predictive distribution varies in response to new data, namely on the predictive updates.   

{\em Discussion.} 
Although 
ours are \lq \lq probability one results", they give insights on frequentist properties of the posterior distribution, from a novel perspective explicitly related to the behavior of the predictive learning rule.  
Roughly speaking, our results suggest that frequentist consistency at $F_0$, and frequentist coverage,  can be read as a problem of \lq \lq efficiency" of the predictive distribution: if the data are generated as i.i.d. from $F_0$, the predictive distribution, that is, the adopted learning rule, should be able to  \lq efficiently' learn that. 
As discussed in \cite{fortini2023}, the predictive updates should balance the convergence rate with 
a proper \lq \lq learning rate'': if $P_n$ converges quickly, with predictive updates that quickly decrease to zero, at step $n$ we would be rather sure about the limit $\tF$ of $P_n$, reflected in small uncertainty (small variance $V_{n,t}$) 
in the posterior distribution of $\tF$ and narrow credible intervals.  On the other hand, 
very 
small predictive updates could reflect poor learning; the extreme case being a degenerate predictive distribution $P_n=P_0$ for any $n$, that converges immediately but does not learn from the data.  An open problem we see is thus to explore conditions under which the predictive rule efficiently balances convergence and learning properties and provides asymptotic credible intervals for $\tF(t)$ with good frequentist coverage.

\section{Methods for predictive constructions} \label{sec:methodsConstruction}
 The reader may be fairly convinced that predictive modeling is conceptually sound;  but may still be concerned that is it difficult to apply in practice. An interpretable statistical model, when possible, incorporates valuable information, and sounds more natural. Moreover, while there is wide literature on prior elicitation, methodological guidance on \lq \lq predictive elicitation" is quite fragmented. 
 The aim of this section is to trace 
 some available methodology, and provide a few examples.
The methods include the notion of predictive sufficiency, that reconciles  predictive modeling to parametric models; and the different notion of {\em sufficientness}, that generally leads to nonparametric constructions - a point that seems overlooked; 
and predictive constructions based on stochastic processes with reinforcement. 
Most of the examples we provide come from Bayesian nonparametric statistics and machine learning, where the predictive approach allows to overcome difficulties in assigning a prior law on infinite-dimensional random objects and has indeed been the basis of vigorous theoretical and applied developments. 

\subsection{Constraints on the form of point predictions} \label{sec:constraints}
Basically all predictive constructions make some assessment on the form of the predictive distribution.  If a parametric model has been already chosen, it may be enough to restrict the class of point predictions $E(X_{n+1} \mid x_{1:n})$. Diaconis and Ylvisaker's \citep{diaconisYlvisaker1979} characterization of conjugate priors for models in the natural exponential family (NEF) is possibly the most classic example. 

\begin{example} ({\em Conjugate priors for the NEF}.)\label{ex:diaconisYlvisaker} 
Let 
 $X_i \mid \theta \iidsim p(x \mid \theta) = e^{x^T\theta- M(\theta)},$ 
where $p(\cdot \mid\theta)$ is a probability density function on $\mathbb R^k$ with respect to a dominating measure $\lambda$ whose support contains an open interval of $\mathbb R^k$, and  $M(\theta)=\ln \int e^{x^T\theta}d\lambda(x)$, for  $\theta\in \Theta=\{s\in \mathbb R^k:M(s)<\infty\}$. 
Because the model is given, the predictive rule characterizes the prior distribution $\pi$ of $\ttheta$, which is assumed to be non degenerate. Let $\tilde\mu=E(X_1\mid\ttheta)$ denote the mean vector parameter, which is also the point prediction: $E(X_2\mid X_1)=E(\tilde\mu\mid X_1)$. 
Diaconis and Ylvisaker (\cite{diaconisYlvisaker1979}, Theorem 3) prove that if $E(\tilde\mu\mid X_1)=aX_1+b$ with $a\in\mathbb R$ and $b\in \mathbb R^k$, 
then $a\neq 0$ and the prior density on the natural parameter $\tilde \theta$ is the conjugate prior 
$\pi(\theta)=c \exp(a^{-1}b^T\theta  -a^{-1}(1-a)M(\theta)).
$
This result does not apply to discrete distributions in the NEF, since the support of the dominating measure does not include an interval of $\mathbb R^k$. For discrete {\em univariate} distributions they give an analogous characterization under the assumption $\Theta=(-\infty,\theta_0)$ with $\theta_0<\infty$. The characterization for the Poisson distributions was already known. 
$\square$\end{example}

\begin{example}\label{example:urnBinary} ({\em Conjugate prior for binary data}).  
Let  $X_i\mid\tilde p=p \iidsim \mbox{Bernoulli}(p)$. Diaconis and Ylvisaker \citep{diaconisYlvisaker1979} prove that, if $ E(\tilde p\mid X_1,\dots,X_n)$ - that coincides with $E(X_{n+1}\mid X_1,\dots X_n)$ - is linear in $\overline X_n$ for every $n$, 
then the prior 
on $\tilde p$ is the conjugate Beta distribution.
The result extends to the characterization of the Dirichlet as the unique family of distributions allowing linear posterior expectation for multinomial observations. 
$\square$ \end{example}

\subsection{Predictive sufficient statistics and parametric models} \label{subsec:sufficiency}
A natural 
tool for predictive 
 elicitation is predictive sufficiency. 
For exchangeable sequences $(X_n)_{n\geq 1}$, the predictive distribution $P_n$ is a function of the entire empirical distribution $\hF$. In other words, the empirical distribution is a sufficient summary of $(X_1, \ldots, X_n)$ 
for prediction of future observations, which is an immediate consequence of exchangeability. In many applications, it is natural to think that a summary $T_n=T(\hF)$ of $\hF$ is sufficient, i.e. that the predictive distribution is a function of $T_n$. The statistic $T_n$ is said to be sufficient for prediction or {\em predictive sufficient.}

Predictive sufficiency has been investigated by several authors from the 1980's;  see the book by Bernardo and Smith (\cite{bernardoSmith1994}, Sect 4.5) and Fortini {\em et al.} \cite{fortini2000}, which includes extensive references. 
Related notions of sufficiency have been studied by Lauritzen (\citep{lauritzen1984}, \cite{lauritzen1988}) and Diaconis and Freedman \cite{diaconisFreedman1984-sufficiency};  
Schervish (\cite{schervish1995-book}, Sect 2.4) gives a review. Several results, and relations with classical and Bayesian sufficiency,  are given in   \cite{fortini2000}. 

The assumption of a predictive sufficient statistic is strictly connected with the assumption of a parametric model.
Informally, if the predictive distribution depends on the data through a predictive sufficient statistic $T_n=T(\hF)$ - 
 expressed, with an abuse of notation,  as  $
 \PP(X_{n+1}\in\cdot \mid X_1,\dots,X_n)=P_n(\cdot \mid T(\hF))
$ - 
then we can expect that, under conditions on 
$T$,   
$$T_n\equiv T(\hF) \rightarrow T(\tF) \equiv \tilde{\theta},
$$ 
(because $\hF \rightarrow \tF$); and, under conditions on $P_n$ as a function of $T_n$, 
$$P_n(\cdot \mid T_n) \rightarrow F(\cdot\mid  \tilde{\theta}) 
$$
for a function $F$. That is, the statistical model (which is the limit of $P_n$)
has a parametric form $F(\cdot\mid \tilde{\theta})$ where the parameter $\tilde{\theta}=T(\tF)$ is the limit of the predictive sufficient statistic. 
This is the content of next theorem. A more general result, but technically more involved, is in \cite{fortini2000}, Theorem 7.1.

\begin{theorem}\label{th:sufficient}
    Let $(X_n)_{n\geq 1} \sim \PP$ be an exchangeable sequence with directing random measure $\tF$. 
Assume that there is a predictive sufficient statistic $T_n=T(\hF)$, 
where 
     $T: {\cal M} \rightarrow \mathbb T\in\mathcal B(\mathbb{R}^k)$ 
     is a continuous function defined on a measurable set $\cal M$ of probability measures such that $\mathbb P(\tilde F\in \mathcal M)=1$. For every $n\geq 1$, let 
     $q_n(\cdot, t)=\PP(X_{n+1} \in \cdot \mid T(\hat{F}_n)=t)$, $t \in \mathbb T$. \\
        If, for every $A$ with $P_0(\partial A)=0$, the functions $(q_n(A, \cdot))_{n\geq 0}$ are continuous on $\mathbb T$, uniformly in $t$ and $n$,       
               then there exists a function $F$ such that 
    $\tF(\cdot)=F(\cdot\mid  \tilde{\theta})$, where $\tilde{\theta}=T(\tilde F)$ is the $\PP$-a.s. limit of $T_n$. 
\end{theorem} 
 The continuity assumptions in the theorem seem reasonable as a \lq robustness' requirement expressing the idea that small changes in the value of the predictive sufficient statistic $T_n$ do not lead to abrupt changes in the prediction. 
The proof is in Section A3 
 of the Supplement \citep{supplement}.

\begin{example} 
Consider a Gaussian model
$X_i \mid \mu  \iidsim \Norm(\mu, \sigma^2)$, with 
$\tilde \mu \sim \Norm(0,1)$ and known variance $\sigma^2$, for simplicity equal to one. 
Take $\mathcal M$ as the set of probability measures with finite first moment, $\mathbb T=\mathbb R$ and $T(m)=\int x \,dm(x)$, for $m\in\mathcal M$. 
The conditions of \cref{th:sufficient} hold. First,
$
E(\int |x|\tilde F(dx))=\int |x| dP_0(x)<+\infty,
$
which implies that $\tilde F\in\mathcal M$, $\PP$-a.s. The function $T$ is continuous on $\mathcal M$
and, for every $A$, the evaluation on $A$ of $q_n(\cdot, t)=\Norm(n/(n+1) t, (2+n)/(1+n))$ 
is continuous in $t$, uniformly with respect to $t$ and $n$. 
 $\square$\end{example}

Theorem \ref{th:sufficient} gives sufficient conditions under which the statistical model is parametric. 
Stronger conditions are needed if we want to obtain a dominated model. 
\begin{proposition}
Under the assumptions of \cref{th:sufficient}, and 
\begin{itemize}
\item[{\rm i)}]  the predictive distributions $P_n$ are absolutely continuous w.r.t. a dominating measure $\lambda$,
\item[{\rm ii)}] with probability one, the sequence $(P_n)_{n \geq 0}$ converges to the directing random measure $\tilde F$ in  total variation, 
\end{itemize}
then the statistical model is dominated, i.e. $\PP$-a.s., $\tilde F(\cdot)=F(\cdot\mid \ttheta)$,  
with  $F(\cdot\mid\theta) $  absolutely continuous with respect to  
$\lambda$ for every $\theta$.
\end{proposition}
The proof follows from Theorem 1 in  \cite{BertiPratelliRigo2013}, which shows that the  conditions i) and {\rm ii)} are necessary and sufficient for the random directing measure $\tF$ to be absolutely continuous w.r.t. $\lambda$. By Theorem \ref{th:sufficient}, $\tF$ has parametric form $F(\cdot \mid \ttheta)$, and because the limit of $P_n$ is unique almost everywhere, we have the conclusion. 

\subsection{Predictive \lq \lq sufficientness'' } \label{sec:sufficientness}
A different concept is predictive \lq \lq sufficientness'' \citep{zabell2005}. 
The term \lq sufficientness' was used by Good \citep{good1967} with reference to the work by W. E. Johnson \citep{johnson1932}. Zabell \citep{zabell1982} notes that Good \citep{good1965-book} initially used \lq sufficiency' but switched to \lq sufficientness' to avoid confusion with the usual notion of sufficiency. Here, there is no predictive sufficient statistic beyond the empirical distribution; however, for every set $A$, the probability that a future observation takes value in $A$ is assumed to depend only on $\hat F_n(A)$. 
In principle, only sufficientness assumptions of the kind above 
are made;  it is however assumed that $(X_n)_{n\geq 1}$ is exchangeable, which introduces constraints on the permissible analytic form of $P_n$ and may identify it.

Interestingly, since the entire empirical distribution is needed for prediction of future observations, we expect that,  if no further restrictions 
are made beyond exchangeability and sufficientness, this type of predictive constructions leads to a {\em nonparametric} model.

\begin{example} ({\em Sufficientness  characterization of the Dirichlet conjugate prior for categorical data}). 
Consider an exchangeable sequence $(X_n)_{n\geq 1}$ of categorical r.v.'s with values in
$\{1, \ldots, k\}$ with $k>2$, finite. With the notation as in \cref{example:multinomial}, 
\begin{equation}\label{eq:discreteNP} X_i \mid (\tilde{p}_1, \dots, \tilde{p}_k) \iidsim   
\left\{
\begin{array}{l}
1 , \ldots ,k\\
\tilde{p}_1, \dots, \tilde{p}_k.
\end{array}
\right.
\end{equation}
No parametric form is imposed on the masses $(\tilde p_1, \ldots, \tilde{p_k})$; in this sense, this is a \lq \lq nonparametric" setting.
Since the sequence $(X_n)_{n\geq 1}$ is exchangeable, the predictive distribution depends on the empirical frequencies $(n_1, \ldots, n_k)$, i.e. $\PP(X_{n+1}=j \mid x_{1:n})=\PP(X_{n+1}=j \mid n_1, \ldots, n_k)$. 
The sufficientness postulate 
states that the predictive probability of  $X_{n+1}=j$ only depends on $n_j$, 
\begin{equation} \label{eq:zabell-sufficientness}
    \PP(X_{n+1}=j \mid x_{1:n})= f_{{ n,j}}(n_j) , \quad j=1, \ldots, k. 
\end{equation}
We stress that, to provide the predictive probabilities for {\em all}  $j$, the entire vector of empirical frequencies is  needed. 

Formally developing an argument by \citep{johnson1932}, Zabell \citep{zabell1982} proves that the sufficientness assumption (\ref{eq:zabell-sufficientness}),   
together with $\PP(X_1=x_1,\dots,X_n=x_n)>0$ for every $(x_1,\dots,x_n)$, implies that $f_{{n,j}}(n_j)$ is linear in $n_j$, and more specifically, that, if the $X_i$ are not  independent,  there exist positive constants $(\alpha_1, \ldots, \alpha_k)$  
such that 
\begin{equation} \label{eq:zabell-dirichlet}
\PP(X_{n+1}=j \mid n_j)=\frac{\alpha_j+n_j}{\alpha+n}, 
\end{equation}
where $\alpha=\sum_{i=1}^k \alpha_i$. In turn, this  allows to obtain the expression of all the moments of the prior distribution, which are shown to characterize the Dirichlet($\alpha_1, \ldots, \alpha_k$) distribution as the prior for  $(\tilde{p}_1, \ldots, \tilde{p}_n)$.
\cite{zabell1982} also includes results for finite exchangeable sequences.
$\square$\end{example}

Johnson's sufficientness postulate can be extended to the case of r.v.'s taking values in a general Polish space $\mathbb  X$. 
Consider $(X_n)_{n\geq 1}$ exchangeable and assume that  for any $n \geq 1$, the predictive rule states that for any measurable set $A$ 
\begin{equation} \label{eq:suff-DP}
P_n(A)= \PP(X_{n+1} \in A \mid \hat F_n(A)).
    \end{equation}
Since $(X_n)_{n\geq 1}$ is exchangeable, there exists $\tilde F$ such that $X_i \mid \tF \iidsim \tF$. Again, the entire empirical distribution is needed to obtain the predictive distribution, thus we expect to characterize a nonparametric prior on the random distribution $\tF$. 
This is indeed the case.
  \begin{proposition}\label{prop:suffDP}
Let $(X_n)_{n \geq 1}$ 
be an exchangeable sequence and assume that $X_1 \sim P_0$ and, for any $n \geq 1$, the predictive distribution satisfies (\ref{eq:suff-DP}).   
If the $X_i$ are not independent, then the directing random measure $\tF$ 
has a Dirichlet process distribution with parameters $(\alpha , P_0)$ for some $\alpha >0$, denoted $\tF \sim$ DP$(\alpha , P_0)$.
 \end{proposition} 
The proof of Proposition \ref{prop:suffDP} is in Section A3 
of the Supplement \citep{supplement}. 
This result seems new. Doksum (\cite{doksum1974}, Corollary 2.1) proves that the Dirichlet process is the only \lq non trivial' process such that the posterior distribution of $\tF(A)$ given $x_{1:n}$ 
only depends on the number $n_A$ of observations in $A$ (and not on where they fall within or outside $A$). This implies that the predictive distribution of $X_{n+1}$ given $x_{1:n}$ 
only depends on $n_A$; but the latter is a weaker condition. The proposition above shows that it still implies that $\tF$ is a Dirichlet process.
Other characterizations of the Dirichlet process through sufficientness use the additional  assumption that the predictive distribution has a specific linear form as e.g. in \cite{lo1991} or, equivalently, assume that the $X_i$ are categorical random variables. 
Actually, the  sufficientness postulate (\ref{eq:suff-DP}) is reasonable 
only for categorical random variables 
(for continuous data, for example, one would not fully exploit the information in the sample). 

\medskip
A number of nonparametric priors are characterized by forms of predictive sufficientness. 
Zabell \cite{zabell1997} characterizes the 
two parameter Dirichlet process 
from sufficientness assumptions (see \cref{ex:twoparameters} in \cref{sec:examples}).  
Extensions to the class of Gibbs-type priors \cite{deBlasi-gibbs2015} and to  hierarchical generalizations are given by \cite{bacallado2017}.
Muliere and Walker \cite{walkerMuliere1999-NTR} give a predictive characterization of Neutral to the Right processes \citep{doksum1974}  based on an extension of Johnson's sufficientness postulate. 
Sariev and Savov \cite{sariev2023-sufficientness} provide a sufficientness characterization of exchangeable measure-valued P\'olya urn sequences. 

\subsection{Stochastic processes with reinforcement}
Stochastic processes with reinforcement, originated from an idea by Diaconis and Coppersmith \cite{diaconisCoppersmith1986}, are perhaps the main tool used in Bayesian statistics for predictive constructions. They express the idea that, if an event occurs along time, the probability that it occurs again in the next time increases (is {\em reinforced}). They are of interest in probability and in many areas beyond Bayesian statistics; applications include population dynamics, network modeling (where they are often referred to as preferential attachment rules), learning and evolutionary game theory, self-organization in statistical physics and many more. A beautiful review is given by Pemantle \cite{pemantle2007}.

Urn schemes 
are basic building blocks for random processes with reinforcement.  

\begin{example} ({\em Two-color \Polya urn})
The simplest example is the two color \Polya urn (\cite{Epolya1923}, \cite{polya1931}). 
One starts with an urn that contains $\alpha$ balls, of which $\alpha_1$ are white and the others are black. At each step, a ball is picked at random and returned in the urn along with an additional ball of the same color. Denoting by $X_n$ the indicator of a white additional ball at step $n$, and by $Z_n$, $n \geq 0$, the proportion of white balls in the urn 
before the $(n+1)$th draw, 
we have $\PP(X_1=1)=\alpha_1/\alpha=Z_0$ and for any $n \geq 1$
$$  \PP(X_{n+1}=1 \mid X_1,\dots,X_n)= \frac{\alpha_1 + \sum_{i=1}^n X_i}{\alpha + n}=Z_n.  $$
The two color \Polya urn was proposed as a model for the evolution of contagion. In Bayesian statistics, \Polya sampling is not meant as describing a process that actually evolves over time (such as the spread of contagion), but describes the evolution of information; namely a learning process where the probability that the next observation is white 
is {\em reinforced} as more white balls are observed 
in the sample. It is well known that the sequence $(X_n)_{n\geq 1}$ so generated is exchangeable, and that both the relative frequency $\sum_{i=1}^n X_i/n$ and the proportion of white balls $Z_n$ converge to a random limit $\ttheta \sim \mbox{Beta}(\alpha_1, \alpha-\alpha_1)$. Thus, from the predictive rule we get $X_i \mid \ttheta \iidsim \mbox{Bernoulli}(\ttheta)$ with a conjugate Beta$(\alpha_1, \alpha-\alpha_1)$ prior.
$\square$\end{example}

The celebrated extension to a countable number of colors are \Polya sequences \citep{blackwell1973}, see the following \cref{ex:polyasequence}. Many more exchangeable predictive constructions are based on reinforced stochastic processes; we provide a few notable examples in the next section.  

\subsection{Examples in Bayesian nonparametrics} \label{sec:examples}

\begin{example}{({\em The Dirichlet process})} \label{ex:polyasequence}
In \cref{sec:sufficientness}, we have seen a characterization of the Dirichlet process in terms of sufficientness. The predictive characterization as an extension of \Polya sampling was given by Blackwell and MacQueen \cite{blackwell1973}. For data in a Polish space $\mathbb X$, Blackwell and MacQueen define {\em \Polya sequences} $(X_n)_{n \geq 1}$ as characterized  
by the predictive rule $X_1 \sim P_0$ and for any $n \geq 1$, 
\begin{equation} \label{eq:predDP}
    X_{n+1} \mid X_1,\dots,X_n \sim P_n=\frac{\alpha}{\alpha+n} P_0  + \frac{n}{\alpha+n} \hat{F}_n,
\end{equation}
where $\alpha > 0$. They prove 
that a  \Polya sequence is exchangeable and $P_n$ converges $\PP$-a.s. to a {\em discrete} random distribution $\tF$; moreover,   
 $\tF \sim$ DP$(\alpha, P_0)$. 
 It follows that $X_i \mid \tF \iidsim \tF$, with a DP$(\alpha, P_0)$ prior on $\tF$. 

\Polya sequences can be also described as reinforced urn processes; the interest in such characterization is that it enlightens the link with the theory of random partitions. Indeed, the discrete nature of the Dirichlet process, that follows from (\ref{eq:predDP}), implies that ties are observed in a random sample $(X_1, \ldots, X_n)$ with positive probability. 
This induces a random partition of $\{1, \ldots, n\}$, with $i$ and $j$ in the same group 
if $X_i=X_j$. The characterization 
as a reinforced urn model explicates its probability law. 

Rather than an impractical urn with infinitely many balls, a  proper urn metaphor is the 
Hoppe’s urn scheme (\cite{hoppe1984}, \cite{hoppe1987}), also popularly described  as the Chinese Restaurant Process \cite{aldous1985}. 
Consider sampling from an urn that initially only contains  $\alpha > 0$ black balls.  
At each step, a ball is picked at random and, if colored, it is returned in the urn together with an additional ball of the same color; if black, the additional ball is of a new color. Natural numbers are used to label the colors and they are chosen sequentially as the need arises. 
The sampling generates a process $(L_n)_{n \geq 1}$, where $L_n$ denotes the label of the additional ball returned after the $n$th draw.  Clearly, the sequence $(L_n)_{n \geq 1}$ is not exchangeable.
However, if one \lq paints' it, picking colors, when needed, from a 
 color distribution $P_0$,  then the resulting sequence of colors $(X_n)_{n\geq 1}$ has predictive rule (\ref{eq:predDP}), thus it is a P\'{o}lya sequence with parameters $(\alpha$, $P_0)$. In terms of the  Chinese 
 Restaurant metaphor, where customers enter sequentially in the restaurant and are allocated either in a occupied table, or in a new one,  $L_n=j$ denotes that the $n$ customer 
 is seated at table $j$, and for any $n \geq 1$, the label's configuration $(L_1, \ldots, L_n)$
gives the allocation of customers at tables, representing the random partition;  
then tables are painted at random from the color distribution $P_0$. 

For any $n \geq 1$, 
let $\rho_n=(A_1, \ldots, A_{k_n})$ be the random partition of $ \{1, 2, \ldots, n\}$ 
so generated (where $i \in A_j$ if $L_i=j$,  $k_n$ is the number of colors that have been created, or of the occupied tables, and the $A_j$ are in order of appearance). The probability mass function, or {\em partition probability function} of $\rho_n$ is easily computed from the labels' sampling scheme; if $P_0$ is diffuse, we have 
\begin{equation}
\label{eq:EPPF-DP}
 \PP({\rho}_n=(A_1, \ldots, A_{k_n}))= 
 \frac{\alpha^{k_n}}{\alpha^{[n]}} \, \prod_{j=1}^{k_n} (n_j-1)!
\end{equation}
where $\alpha^{[n]}=\alpha (\alpha+1) \cdots (\alpha+n-1)$ and $n_j$ is the number of elements of $A_j$, $j=1, \ldots, k_n$. See \cite{ewens1972}, \cite{antoniak1974},  \cite{hoppe1984}.  $\square$
\end{example}

The above characterization is an emblematic example of the potential of predictive constructions - in this case, explicating the link with random partitions theory. In Bayesian statistics, the capacity of the Dirichlet process of generating random partitions is leveraged 
for model based clustering in many applications; beyond Bayesian statistics, random partitions, and in particular, {\em exchangeable} random partitions, are of interest in a wide range of fields such as combinatorics, 
genetics, population dynamics. 
The construction of \cref{ex:polyasequence} extends more generally, and we recall here some basic notions that we use in the following examples.

Given an exchangeable sequence $(X_n)_{n\geq 1}$ one   can define a random partition $\rho_n$ 
of $\{1, \ldots, n\}$ by letting $i$ and $j$ be in the same group if $X_i=X_j$. Then we have
\begin{eqnarray} \label{eq:EPPF}
& \PP( \rho_n=(A_1, \ldots, A_{k_n}))
= p(n_1, \ldots, n_{k_n})
\end{eqnarray}
for a symmetric function $p$ of $(n_1, \ldots, n_{k_n})$, where $n_j$ is the number of elements in $A_j$. A partition probability function $p$ so generated is called the {\em exchangeable partition probability function}
(EPPF) derived from the sequence $(X_n)_{n\geq 1}$. More formally, $p$ is defined on the space of sequences ${\bf 
n}=(n_1,n_2, \ldots)$, identifying  $(n_1, \ldots, n_{k_n})$ as ${\bf n}=(n_1, \ldots, n_{k_n}, 0, 0, \ldots)$. Let 
${\bf n}^{j+}$ be defined from ${\bf n}$ by incrementing $n_j$ by $1$. Clearly an EPPF $p$ must satisfy
$$
p(1,0,0,\dots)=1 \quad \makebox{and } \; p({\bf n})= \sum_{j=1}^{{k_n}+1} \, p({\bf n}^{j+}).
$$
In predictive modeling,  
the conditional probability that the next observation $X_{n+1}$ is in group $j$,
given $x_{1:n}$, is 
\begin{equation}
    \label{eq:predEPPF}
p_j({\bf n})=\frac{p({\bf n}^{j+})}{p({\bf n})}\mbox{ provided }p({\bf n})>0,
\end{equation}
for $j=1,\dots, {k_n}+1$.
The concept of EPPF has been introduced in Pitman \cite{Pitman1995-exchRandomPartitions}.  
A fundamental result in the theory of exchangeable random partitions is Kingman's 
de Finetti-like representation theorem for exchangeable random partitions as mixtures of paint-box processes \citep{Kingman1978-representation}; see also Kingman \cite{kingman1980-book}, 
Pitman \cite{Pitman2002-book}, Zabell \cite{zabell2005}.

\begin{example}\label{ex:SSM}{({\em Species sampling priors})}. 
Pitman \cite{pitman1996} defines a class of predictive rules, in the framework of  species sampling, that generalizes Blackwell and McQueen scheme (\ref{eq:predDP}). 
One underlines sequential sampling from a discrete random distribution for categorical data - in species sampling, sequential draws  from a population of species labeled in the order they are discovered with tags $X_j^*$ i.i.d. from a diffuse distributions $P_0$. Here $X_i$ represents the species of the $i$th individual sampled and takes values in the set of tags. 
In a sample $x_{1:n}$,  
one will observe $k_n$ distinct species, labeled $x_1^*, \ldots, x_{k_n}^*$
and the next observation $X_{n+1}$ will either be one of the species already discovered in the sample, or a new one, formalized in the  predictive distribution 
$ P_n(\cdot \mid x_{1:n})=
\sum_{j=1}^{k_n} p_{j,n}(x_{1:n}) 
\delta_{x_j^*}(\cdot) + p_{k_n+1,n}(x_{1:n}) 
P_0(\cdot)$. 
In random sampling,
the sequence $(X_n)_{n \geq 1}$ should be exchangeable, and a necessary condition is that $p_{j,n}$ depends on $(x_1, \ldots, x_n)$
only through the sequence of counts  ${\bf n}=(n_{1}, n_{2}, \ldots)$ (terminating with a string of zeroes)
 of the various species in the sample in the order of appearance.  

 A sequence $(X_n)_{n \geq 1}$ is a {\em species sampling sequence} if it is exchangeable and has a predictive rule of the form 
\begin{equation} \label{eq:predsss}
P_n(\cdot)=\sum_{j=1}^{k_n} p_j({\bf n}) \delta_{X_j^*} (\cdot) + p_{k_n+1}({\bf n}) P_0(\cdot),
\end{equation}
for $n\geq 1$, for a diffuse distribution $P_0$, which is also the law of $X_1$. 
Pitman (\cite{pitman1996}, Theorem 14) shows that exchangeability holds if and only if 
there exists a non-negative symmetric function $p$ that drives the probabilities $p_j(\bf n)$ according to \eqref{eq:predEPPF}.  
Then the EPPF of $(X_n)_{n\geq 1}$ is the unique non-negative symmetric function $p$ such that \eqref{eq:predEPPF} holds and $p(1) = 1$. 

From exchangeability, by Proposition \ref{prop:predconv} we have that, with probability one,  
$P_n$ converges to a random distribution $\tF$; Pitman 
(\cite{pitman1996}, Proposition 11) proves that  the convergence is in total variation norm and $\tF$ has the form 
 \begin{equation} \label{eq:predLimiteSss}
\tF(\cdot)=\sum_{j=1}^{k_{\infty}} p^*_j \delta_{X_j^*}(\cdot) + (1-\sum_{j=1}^{k_{\infty}} p^*_j) P_0(\cdot),
\end{equation}
where $p^*_j = \lim n_j /{n}$ is the random 
limit of the 
relative frequency of the $j$-th species 
discovered, 
the $X_j^*$ are i.i.d. according to $P_0$, independently of the $p^*_j$ and 
$k_{\infty}=\inf\{k : p^*_1+\cdots+p^*_k=1\}$ is the number of distinct values to appear in the infinite sequence $(X_1, X_2, \ldots)$. 

The above results do not provide an explicit description of the distribution of  the weights $p^*_j$, which is however available in remarkable special cases, including the Dirichlet process, that corresponds to $p_j({\mathbf n})= n_{j}/(\alpha+n)$, where $\alpha>0$ is a fixed number;  
the finite Dirichlet process \citep{Ishwaran2001} that assumes $p_j({\mathbf n})=(n_{j}+\alpha/K)/(\alpha+n)$ for $j\leq k_n\leq K$, where $\alpha>0$ and $K\in \mathbb N$ are fixed numbers;  and the two parameter Poisson-Dirichlet process. 
$\square$\end{example}

\begin{example}({\em Two parameter Poisson-Dirichlet process)} \label{ex:twoparameters}
The two parameter Poisson-Dirichlet process, or Pitman-Yor process, introduced in \cite{permanPitmanYor1992} and further studied in \cite{Pitman1995-exchRandomPartitions} and \cite{pitmanYor1997}, can be viewed both as an extension of the Dirichlet process and as the directing random measure  of a specific species sampling sequence  characterized by 
the predictive rule \eqref{eq:predsss} with
\begin{equation}\label{eq:predPD}
    p_{j}(\mathbf n)=\frac{n_j 
    -\theta}{\alpha+n}\;\mbox{ and } \;p_{k_n+1}(\mathbf n)=\frac{\alpha+k_n \theta}{\alpha+n},
\end{equation}
where $\alpha$ and $\theta$ are real parameters satisfying  $0 \leq \theta < 1$ and $\alpha > - \theta$.
As it appears from (\ref{eq:predPD}), the Poisson-Dirichlet process allows a more flexible predictive structure than the Dirichlet process (corresponding to $\theta=0$): the predictive probability of observing a new species at time $n$ depends on both $n$ and the number $k_n$ of distinct species sampled. 

In analogy to \cref{ex:polyasequence}, the sequence $(X_n)_{n\geq 1}$ can be described 
as a {\em generalized Hoppe's urn} \citep{zabell1997} if $\alpha>0$.  
Initially, the urn only contains one black ball of weight $\alpha$, and balls are selected with probabilities proportional to their weights;  
whenever a black ball is selected, it is 
returned into the urn together with {\em two} new balls, one black, having weight $\theta$, and one of a new color, sampled from $P_0$, having weight $1-\theta$. 

As shown by Zabell \cite{zabell1997}, the two-parameter Poisson-Dirichlet process 
is also characterized through sufficientness (\cref{sec:sufficientness}); namely, by 
the sufficientness  of $n_j$ 
and of $k_n$ 
in the predictive probabilities $p_j(\mathbf n)$ 
and $p_{k_n+1}(\mathbf n)$, 
respectively.

For increasing $n$, the predictive distribution $P_n$ converges $\PP$-a.s. to a random measure $\tF=\sum_{j=1}^\infty p^*_j \delta_{X_j^*}$, where the $p^*_j$ have the stick-breaking representation 
$p_j^*=\prod_{i=1}^{j-1}(1-V_i) V_j$, with $V_i \indsim $ Beta$(\alpha+i \theta, 1- \theta)$.
Again, the predictive construction can be exploited to design computational strategies; see e.g. \cite{Bacallado2022-PY}.
$\square$\end{example}

In some examples, the predictive construction does not characterize a novel prior, but explicates the predictive assumptions that are made when adopting a certain (already known) prior law  -- which is clearly important;  
and here is an example of a purely predictive construction, whose de Finetti-like representation and implied prior law was only given 
afterwards. 

\begin{example}[{\em Indian Buffet Process}] 
The Indian Buffet process, introduced by Griffith and Ghahramani \cite{griffithsGhahramani2005}, is a clever and popular 
predictive scheme for infinite latent features problems. Here, exchangeable objects or individuals are each described through a potentially infinite array of features, resulting in an underlying random binary matrix with rows representing the individuals, and an unbounded number of columns, representing the features. 
Specifically, a $1$ in the $[n,k]$ entry of the random matrix indicates that the $n$th individual possesses the $k$th feature. The predictive construction can be illustrated by imagining customers sequentially entering an Indian Buffet restaurant. In this metaphor, customers represent 
individuals, dishes symbolize features, and when a customer selects a dish $z$, it means that the corresponding individual possesses feature $z$. Let $\theta$ be a fixed strictly positive number. The first customer  chooses a Poisson$(\theta)$  number of dishes from a non-atomic distribution 
$F_0$. Then, for $n=1,2,\dots$, the $(n+1)$th customer decides, for each of the $k_n$ dishes already served, whether to take it or not,  
according to its 
popularity, namely she chooses  dish $z$ with probability 
${k_{z,n}}/{(n+1)}$,  where $k_{z,n}$ is the number of 
customers who have already  chosen dish $z$, independently for $z=1,\ldots, k_n$;  
then she chooses a Poisson$({\theta}/{(n+1)})$ number of new dishes, sampling them from $F_0$. 

This construction, which is purely predictive, 
characterizes an exchangeable law for the individuals' features, 
represented as $X_n=\sum_{k=1}^\infty b_{n,k}\delta_{Z_k}$, 
where ${b}_{n,k}=1$ if the $n$th individual possesses feature $Z_k$, and zero otherwise, 
and with the features $(Z_k)_{k\geq 1}$ independently sampled from  $F_0$; and enables  Bayesian learning without an explicit prior law.
Actually, the implied prior law 
was later made explicit \citep{thibauxJordan2007}, and assumes that,
conditionally on a sequence $(p_k)_{k\geq 1}$ of r.v.s  
taking values in $(0,1)$, the $(b_{n,k})_{n,k\geq 1}$ are sampled independently, with $b_{n,k}\sim$ Bernoulli$(p_k)$. In turn, the $(p_k)_{k\geq 1}$ are the points of a Poisson random measure with mean intensity $\lambda(s)=\theta s^{-1}\mathbf 1_{(0,1)}(s)$. 

The Indian Buffet process has been extended for 
allowing different distributions on $(p_k,b_{i,k})_{i,k\geq 1}$ (see \cite{james2017}, \cite{camerlenghi2024} and references therein) or random weights  \cite{berti2015}. 
$\square$
\end{example}

\begin{example} [{\em Predictive constructions for continuous data}] 
\label{ex:pred-continuous}
As already mentioned, the predictive rule (\ref{eq:predDP}) of \Polya sequences is appropriate for categorical data, but as it appears from the underlying sufficientness postulate (\ref{eq:suff-DP}), it is not efficient for continuous data, failing to fully exploit the sample information. A similar remark holds for species sampling sequences. 
Indeed, in Bayesian statistics, the Dirichlet process 
and generally discrete prior laws  are mostly used at the {\em latent}  stage of hierarchical models, where, as already noted by Antoniak \citep{antoniak1974}, their power in generating a random partition is an asset; see e.g. 
\cite{tehJordan2010} and \cite{orbanzRoy2015} for overviews. 
A popular example are Dirichlet process mixture models where, conditionally on a latent exchangeable  sequence $(\ttheta_n)_{n\geq 1}$,  the $X_i$ are independent and the distribution of $X_i$ only depends on $\ttheta_i$, with a slight abuse of notation written as 
\begin{eqnarray} \label{eq:DPmixture-lik}
X_i \mid \tilde \theta_i 
\indsim& k(\cdot \mid \ttheta_i),
\end{eqnarray}
for a kernel density $k$, and   
\begin{eqnarray} \label{eq:DPmixture-prior}
\ttheta_i \mid \tilde{G} &\iidsim& \tG , \quad \mbox{with} \quad \tilde G \sim DP(\alpha, G_0).
\end{eqnarray}
This gives an exchangeable mixture model: 
$$X_i \mid \tilde{G} 
\iidsim f_{\tG}(\cdot) 
= \int k(\cdot \mid \theta) d \tG(\theta).$$ 
The predictive rule of the Dirichlet process induces a parametric model on the random partition of the $\ttheta_i$'s, and $X_i$ and $X_{j}$ are set in the same cluster if $\ttheta_i= \ttheta_{j}$. 
This is a powerful and popular use of predictive rules such as  (\ref{eq:predDP}), which however would not be appropriate as  predictive learning rules at the observation level with continuous data. 

An approach to address this difficulty is to smooth the trajectories generated by discrete priors thus obtaining novel prior laws that almost surely select absolutely continuous distributions; for example, a constructive smoothing of the Dirichlet process 
through Bernstein polynomials was proposed, from an idea of Diaconis, by \cite{petrone1999-Bpols} and extended by \cite{petroneVeronese2010}, who obtained a general class of mixture priors.
However, in these constructions, and more generally in Bayesian mixture models with a discrete prior law on the mixing distribution, the predictive distribution is not analytically tractable, requiring to average with respect to the posterior law on the huge space of partitions (see e.g. \cite{wade2014-scandinavian}).  

A predictive approach may consist in directly smoothing the empirical distribution in predictive rules such as (\ref{eq:predDP}). Recent proposals are {\em kernel-based Dirichlet sequences} \citep{bertiRigo2023-kernel}, that are defined as exchangeable sequences whose predictive distributions 
spread the point mass $\delta_{X_i}$ in (\ref{eq:predDP}) through a probability kernel $K$, as 
$$ 
P_n(\cdot)= \frac{\alpha}{\alpha+n} P_0(\cdot) + \frac{1}{\alpha+n} \sum_{i=1}^n K(\cdot \mid  X_i).
$$
The exchangeability condition imposes that the kernel 
$K$ must satisfy $K(\cdot\mid x)=P_0(\cdot\mid \mathcal G)(x)$ for a sigma-algebra $\mathcal G$ on $\mathbb X$ (
\cite{bertiRigo2023-kernel}, 
\cite{SarievSavov2024}). In particular, 
in their perhaps most natural specification, with $K(\cdot\mid x)\ll P_0$ for every $x\in \mathbb X$, 
 the underlying $\tF$ is a mixture model with kernels having known disjoint support (e.g. 
 a histogram with known bins), see \citep{SarievSavov2024}, Theorem 3.13; which is clearly limited for  statistical applications.

This example hints that exchangeability constraints may be quite restrictive if one wants to have both a tractable predictive rule and some desired modeling features. Here is a predictive construction that is analytically simple, and gives another \lq smoothed version' of (\ref{eq:predDP}).
We start from $P_0(\cdot) =\int K(\cdot \mid  \theta) dG_0(\theta) \equiv F_{G_0}(\cdot)$ 
and recursively update our prediction as 
$$
P_n(\cdot)= (1-\alpha_n) P_{n-1}(\cdot) + \alpha_n 
F_{G_{n-1}}(\cdot \mid X_n),$$
with 
$F_{G_{n-1}}(\cdot\mid X_n) = 
\int K(\cdot \mid  \theta) dG_{n-1}(\theta\mid X_n)$, where $G_{n-1}(\cdot\mid X_n)$ 
denotes the posterior distribution obtained from the prior $G_{n-1}$ and updated based on $X_n$, and 
$G_n(\cdot)= (1-\alpha_n) G_{n-1}(\cdot)+ \alpha_n G_{n-1}(\cdot \mid X_n)$; 
and the $\alpha_n$ are real numbers in $(0,1)$  satisfying $\sum_{n=1}^\infty \alpha_n=+\infty$ and $\sum_{n=1}^\infty \alpha_n^2<+\infty$.
(We may recognize \lq Newton's algorithm'  
\cite{newtonZhang1999}, popularly used for fast computations in Dirichlet process mixture models). 
This predictive rule does not characterize an exchangeable sequence $(X_n)_{n \geq 1}$; 
it is however a martingale and preserves exchangeability asymptotically. More specifically, it can be shown (\citep{fortini2012hierarchical}) that $P_n$ converges to a mixture model $F_{\tG}(\cdot)$ with a novel prior law on $\tilde{G}$, as we will expand in Section \ref{sec:algorithms}. 
$\square$\end{example}

Further remarkable examples, among many, include the class of {\em reinforced urn processes} (\cite{walkerMuliere1997-betaStacy}, \cite{muliereSecchiWalker2000}); see Example \ref{ex:reinforced-Hoppe}), 
and constructions aimed at addressing the rigidity of the global clustering induced by the predictive structure (\ref{eq:predDP}) of the Dirichlet process in the case of multivariate random distributions; 
for example,  
\cite{wade2011} obtain a nested clustering for multivariate data 
characterized by an {\em enriched} Hoppe's urn scheme. There are many more predictive constructions based on the idea of reinforcement, 
that characterize forms of partial exchangeability, as we introduce in the next section. 

\section{Partial exchangeability for more structured data} \label{sec:partial exchangeability}

As seen, exchangeability in Bayesian statistics is the natural predictive requirement in homogeneous repeated trials;  but of course data may be much more complex. Still, in many cases the data show  forms of symmetry, such that exchangeability assessments, judging that the individuals' labels in some data sub-structures do not bring any information for prediction, are still natural. 
In this section we review the concept of {\em partial exchangeability}, i.e. invariance under a group of permutations. For the sake of space,
we focus on the main concepts and on 
de Finetti-like representation theorems that again justify the Bayesian inferential model  
from predictive assumptions. The predictive characterization in Theorem \ref{th:partial-by-pred} is new.  
We start with the notion of partial exchangeability in the sense of de Finetti and a point we will underline 
is that 
other forms of partial exchangeability are ultimately related to it.

\subsection{de Finetti's partial exchangeability} \label{sec:partial}

A first notion of partial exchangeability was introduced by de Finetti in \cite{deFinetti1937-scparziale}. 
 It is interesting to report some excerpt from this, perhaps less known, work by de Finetti, as, beyond historical interest, it clearly shows what are the applied contexts that suggest a partial exchangeability assessment. 
de Finetti \citep{deFinetti1937-scparziale} refers to replicates of trials of different {\em types}. 
\begin{quote} 
Exchangeability can still be considered, but specifying that  the trials are divided into a certain number of types, 
and what is judged exchangeable are the events of the same type.
\end{quote}
As a simple example, he considers 
tossing two coins. If the two coins look exactly alike, one may judge all tosses as exchangeable; 
at the opposite extreme, if the coins are completely different, one would consider the corresponding tosses as two separate exchangeable sequences, completely independent of each other. However, 
if the coins look almost alike, but not to the point of considering them exchangeable,  then 
\begin{quote}
observations of the tosses of one coin will still be capable of influencing, although in an {\em less direct} manner, our probability judgment regarding the tosses of the other coin.
\end{quote}
Again from \cite{deFinetti1937-scparziale}:
\begin{quote}
One can have any number of types of trials,
\end{quote}
for example, different coins, 
or tosses of one coin by two different people, or under different conditions of  
temperature and atmospheric pressure.
\begin{quote} 
    If the types are in a countable or continuous set, prediction would typically refer to a new type; thus, information will exclusively be {\em indirect}. 
\end{quote}(de Finetti's note \cite{deFinetti1937-scparziale} includes several more examples, e.g. in insurances and in treatments' effects and debatable causality).
In the Bayesian literature, partial exchangeability in the sense of de Finetti is usually referred to 
random sampling in
parallel experiments; as we see, it refers more generally to fixed-design regression where the \lq types' are induced by covariates. 
As in the coins example, experiments are run independently, nevertheless 
each experiment brings information on the other ones; and because information is described through probability (see \cref{sec:introduction}), the joint probability law will assume a form of dependence across the  experiment-specific samples, i.e. of sharing information across experiments in prediction. 

Formalizing, consider a family of sequences $(X_{n,j})_{n \geq 1}$ of r.v.'s 
where $X_{n,j}$ describes the $n$th observation of type $j$, $j=1, \ldots, M$; $M$ can be finite or infinite, and the types may be  taken from a continuum of types. 
For more compact notation, we may arrange them in an array $[X_{n,j}]_{n \geq 1, j=1, \ldots, M}$.  
The family of sequences $[X_{n,j}]_{n \geq 1, j=1, \ldots, M}$ is {\em partially exchangeable in the sense of de Finetti} 
if its 
probability law is invariant under separate finite permutations 
within each column; that is,  if 
$$
[X_{n,j}]_{ n\geq 1 
, j=1,\dots,M
} \stackrel{d}{=} [X_{ {\sigma}_j(n), j}]_{n\geq 1,j=1,\dots,M}
 $$
for any finite permutation 
${\sigma}_j$, $j=1,\dots,M$. 
Roughly speaking, observations are exchangeable inside each experiment, but not across experiments. Aldous (\cite{aldous1985}, page 23) refers to this symmetry property as {\em exchangeability over $V$}. A sequence $(Y_n)_{n\geq 1}$ is exchangeable over $V$ if $(V, Y_1, Y_2, \ldots) \stackrel{d}{=} (V, Y_{{\sigma(1)}}, Y_{{\sigma(2)}}, \ldots)$ for any finite permutation ${\sigma}$. 
Partial exchangeability 
corresponds to each sequence $(X_{n,j})_{n\geq 1}$ being exchangeable over all the others, collected as $V_j$.

The representation theorem extends to 
partially exchangeable families of sequences.  
\begin{theorem}[{Law of large numbers and de Finetti representation theorem for  partially exchangeable sequences
}] 
Let $[X_{n,i}]_{n \geq 1, i=1, \ldots, M}$ $\sim \PP$ be a partially exchangeable array in the sense of de Finetti. Then:
\begin{itemize}
    \item[{\rm i)}] For $n_1, \ldots, n_M \rightarrow \infty$, the vector of the marginal empirical distributions $(\hat{F}_{n_1}, \ldots, \hat{F}_{n_M})$ converges weakly to a vector of random distributions $(\tF_1, \ldots, \tF_M)$, $\PP$-a.s.;
    \item[{\rm ii)}] 
For any $n \geq 1$ and measurable sets $A_{i,j}$,
\begin{align*}
&\PP( \cap_{j=1}^M (X_{1,j} \in A_{1,j}, \ldots, X_{n_j,j} \in A_{n_j,j}) ) \\
& \quad =\int \prod_{j=1,\dots,M}\prod_{i=1,\dots,n_j} F_j(A_{i,j}) d\pi(F_1,\dots,F_M),
\end{align*}
where $\pi$ is the joint probability law of $(\tF_1, \ldots, \tF_M)$.     
\end{itemize}
\label{th:representation-partial}
\end{theorem}
A proof is in \citep{aldous1985}, pp. 23-25. The representation {\rm ii)} says that  conditionally on $(\tF_1, \ldots, \tF_M)$,  the sequences $(X_{n,j})_{n \geq 1}$ are independent and, within sequence $j$,   the $X_{n,j}$  
are i.i.d. according to $\tF_j$. That is, 
a de Finetti-partially exchangeable array is obtained by first picking $(F_1, \ldots, F_M)$ from a {\em joint} prior distribution  
and then for each $j=1, \ldots, M$ picking $X_{n,j} \iidsim F_j$, independently for different $j$. 

\begin{example}[{\em Hierarchical models}]\label{ex:hierarchical-model}
Consider random samples $(X_{1, j}, \ldots, X_{n_j,j})$, $j=1,\dots,M$, from $ M$ independent parallel experiments, say of binary r.v.'s with experiment specific means $\theta_j$,  
 and the classic problem of estimating the mean vector $(\theta_1, \ldots, \theta_M)$. This problem is also rephrased (e.g. in \cite{efronHastie-CASI}) as predicting $X_{n_j+1, j}$ 
in each experiment.
Bayesian hierarchical models  are a powerful tool for  borrowing strength across experiments and for shrinkage. In this example, a basic hierarchical model regards the parameters as r.v.'s $\ttheta_j$, sampled from a latent distribution,  and assumes a hierarchical structure as follows
\begin{align*}
\ttheta_j \mid \lambda \iidsim \pi(\cdot \mid \lambda), & \quad \mbox{with} \; \lambda \sim h(\cdot),\\ 
(X_{1,j}, \ldots, X_{n_j,j}) \mid \theta_1, \ldots, \theta_M &\iidsim \; \mbox{Bernoulli}(\theta_j), 
\end{align*}
independently across $j$ (here $\pi$ and $h$ denote densities). The theoretical justification of this model comes from the assessment of partial exchangeability of the sequences $(X_{n,j})_{n\geq 1}$, 
and of exchangeability of the experiments.  By partial exchangeability, the observations are exchangeable inside each experiment, but not across them; and  
the sequences $(X_{n,j})_{n \geq 1}$ are only {\em conditionally} independent given $(\ttheta_1, \ldots, \ttheta_M)$, which naturally implies sharing information. 
The dependence across experiments is introduced through the joint 
prior law of $(\ttheta_1, \ldots, \ttheta_M)$ - here, the joint density 
$\pi(\theta_1, \ldots, \theta_M) = \int \prod_{j=1}^M \pi(\theta_j \mid \lambda) h(\lambda) d \lambda$. 
$\square$ 
\end{example}
\begin{example}
In hierarchical models as above, the prior law $\pi$ expresses the judgement that the $\ttheta_j$ - informally, the experiments - are exchangeable. 
But, more generally, the groups may be induced by covariates, or refer to time or space, etc., and the prior would express other forms of dependence. 
For example, consider clinical trials where the outcome is binary (tumor shrank or not), run in different hospitals, with patients receiving the same treatment in all hospitals. Here one would judge that the hospitals' labels do not bring information, that is, the hospitals (the corresponding model parameters $\ttheta_j$'s) are exchangeable; as in the example above. 
Now suppose that different treatments, say different dosages $z_j$, are administrated in different hospitals. 
Then  the groups' labels are 
relevant,  and the prior will not treat the  $\ttheta_j$'s as exchangeable, but
will incorporate the effect of the covariate;
for example, 
 express the idea that $\ttheta_j$ and $\ttheta_k$ are similar if the dosages $z_j$ and $z_k$ are close.

With no replicates inside the groups and no random effects - i.e. in a basic fixed-design regression context where the probability of success is $\ttheta_j=g(z_j; \tilde{\beta})$ for a known $g$ and unknown $\tilde{\beta}$ - partial exchangeability reduces to conditional independence of the $X_{1,j}$'s given $\tilde{\beta}$, with dependence across $j$ modeled through the regression function.$\square$
\end{example}

Marginally, the result of Theorem \ref{th:representation-partial} is not surprising, because each sequence $(X_{n,j})_{n \geq 1}$ is exchangeable and one obtains the marginal directing random measure $\tF_j$ (the statistical model and the prior for experiment $j$) as seen in \cref{sec:exch}; in particular, from 
\begin{equation} \label{eq:marginal-pred} 
P_{n,j} (\cdot) \equiv \PP(X_{n+1,j}\in \cdot\mid X_{1,j},\dots,X_{n,j}) \rightarrow \tF_{j}(\cdot).
\end{equation}
But this is not enough: the theorem characterizes the {\em joint} distribution (the joint prior law) of $(\tF_1, \ldots, \tF_M)$. 

It is this joint distribution that induces probabilistic dependence 
across the individual 
sequences,  i.e. {\em borrowing strength} 
in prediction. 
As in  \cref{ex:hierarchical-model}, 
rather than the marginal predictive distribution $P_{n,j}$ as in 
(\ref{eq:marginal-pred}), 
a more interesting predictive distribution refers to future results in experiment $j$ given past observations therein  {\em and} observations in all the related experiments. Aldous's notion of exchangeability over $V$ is particularly suited.   Let  
$V=[(X_{n, i})_{n \geq 1; i=1, \ldots, M; i \neq j}]$ collect the observations in all the  experiments but the $j$th. Then, with $\PP$-probability one, 
\begin{align} \label{eq:pred-V}
& \lim_n \PP(X_{n+1, j} \in \cdot \mid\! X_{1,j},\dots, X_{n,j}, V)   \\
&= \lim_n \PP(X_{n+1, j} \in \cdot \mid\! (X_{k,i})_{k\leq n,i=1,\dots,M})\!  =\! \tF_j(\cdot).
\nonumber
\end{align}
For a 
proof, see  \cite{aldous1985}. Informally, 
$V$ does not carry additional information only in the limit, when the experiments become independent. 

Note that the rows $([X_{n,1}, \ldots, X_{n,M}])_{n \geq 1}$ of a de Finetti partially exchangeable array are an exchangeable sequence,  with directing random measure $\tF$ on the product space $(\mathbb X_1 \text{\Large $\times$} \cdots \text{\Large $\times$} \mathbb X_M)$ that assumes independent components, i.e. $\tF=\text{\Large $\text{\Large $\times$}$}_{j=1}^M \tilde F_j$.
This implies that the relationship between variables in distinct columns of the array $[X_{n,j}]_{n\geq 1,j=1,\dots,M}$ 
is solely 
driven by the probabilistic link between the marginal directing random measures. Again, the sequences do not {\em physically} interact. 

Also note that $\PP(X_{n+1,j} \in \cdot \mid (X_{m,i})_{m\leq n,i=1,\dots,M}) = 
E(\tF_j(\cdot) \mid(X_{m,i})_{m\leq n,i=1,\dots,M})$ and, 
as shown in equation (\ref{eq:pred-V}), 
approximates $\tF_j$ for $n$ large. Since 
the sequence $(X_{n,j})_{n \geq 1}$ is exchangeable, an alternative approximation of $\tF_j$ is provided by the predictive distribution
$P_{n,j}(\cdot)= 
E(\tF_j(\cdot) \mid X_{1,j},\dots,X_{n,j})$, that is only based on the past observations in experiment $j$. However, the latter uses less information,  resulting in a less efficient approximation:
\begin{align*}
& E\Bigl(\bigl(\tilde F_j(A)-\PP(X_{n+1,j}\in A\mid X_{1,j},\dots,X_{n,j})\bigr)^2\Bigr) \\
& \geq
E\Bigl(\!\bigl(\tilde F_j(A)\!-\!\PP(X_{n+1,j}\in A\mid \!\! (X_{k,i})_{k\leq n,i=\!1,\dots,M}
\bigr)^2\!\Bigr).
\end{align*}
If the sequences $(X_{n,j})_{n\geq 1}$ are independent, both methods yield the same result; 
there is  no gain of information in considering the entire array.

\medskip
The predictive characterization of exchangeability of \cref{th:exch-by-pred} can be extended to de Finetti partial exchangeability. 
\begin{theorem}\label{th:partial-by-pred}
 A family of sequences 
 $[X_{n,j}]_{n \geq 1, j=1, \ldots, M}$
    is partially exchangeable in the sense of de Finetti 
    if and only if for every finite $k\leq M$  and every $n\geq 0$, the following conditions hold: 
    \begin{itemize}
        \item[{\rm i)}] For every 
        measurable sets $A_1,\dots,A_k$ and every $i=1,\dots,k$
        $$
        P_n(A_1\times\dots\times A_k\mid (x_{m,j})_{m\leq n,j\leq k}
          )
        $$
                 is symmetric in $(x_{1,i},\dots,x_{n,i})$;
                        \item[{\rm ii)}] 
                The set function that maps $\{A_j,B_j:j\leq k \}$ into 
                \begin{align*}
                                \int_{A_1\times\dots\times A_k}\!\!\!\!\!\!\!\!\!\!\!\!\!\!\!\!\!\!\!\!
                                P_{n+1}&(B_1\times\dots\times B_k\mid (x_{m,j})_{m\leq n+1,j\leq k})\\
                                &dP_n((x_{n+1,1},\dots,x_{n+1,k})
              \mid (x_{m,j})_{m\leq n,j\leq k})
                    \end{align*}
                              is symmetric in $(A_{i},B_{i})$ for every $i=1,\dots,k$,
 \end{itemize}
 where $P_n$ is to the conditional distribution of $(X_{n+1,j})_{j\leq k}$, given $(X_{m,j})_{m\leq n,j\leq k}$ and $P_0(\cdot\mid  (x_{m,j})_{m\leq 0,j\leq k})$ is meant as $P_0$.
\end{theorem}
The proof is provided in Section A4 
of the Supplement \citep{supplement}. 
The predictive characterization in \cref{th:partial-by-pred} is natural when at each time $n$, a new observation is made for each type. 
In fact, de Finetti's  partial exchangeability can be described as invariance of the law of a sequence $(X_1,X_2,\dots)$ to the permutations acting  separately on $M$ groups of random variables,  forming  a partition of $(X_n)_{n\geq 1}$. In this context, a predictive characterization of partial exchangeability 
should account for the structure of the partition into groups, likely in a nontrivial way.

\medskip

Statistical applications
are broad;  hierarchical models are  one of the key strengths of Bayesian statistics, and great flexibility in sharing information in prediction is enabled through the prior law $\pi$. 
While the choice of $\pi$ in parametric models is a long studied problem,   
defining a nonparametric prior on the vector of random distributions  $(\tF_1, \ldots, \tF_M)$ 
has 
posed challenges; yet, a wealth 
of proposals is nowadays available, 
many of which are defined through, or benefit from, predictive characterizations. 

\begin{example}[{\em Hierarchical Dirichlet process}] \label{ex:HDP} 
The hierarchical Dirichlet process has been introduced in \cite{Teh2006} to model shared 
clusters among groups of data.  
For example, consider the problem of modelling shared topics  in a 
corpus of $M$ documents, where a \lq\lq topic'' induces a multinomial distribution over 
the words of a given dictionary, and a document $j$  is defined as an unordered - exchangeable - sequence  of words $(X_{n,j})_{n\geq 1}$. 
 For each document $j$,  we have a latent sequence of topics $(\ttheta_{n,j})_{n \geq 1}$, and $X_{n,j} \mid \ttheta_{n,j}\sim k(\cdot \mid \ttheta_{n,j})$. The family of sequences $(\ttheta_{n,j})_{n\geq 1}$  
 for $j=1, \ldots, M$ is assumed to be partially exchangeable, thus conditionally independent given the vector $(\tG_1, \ldots, \tG_M)$ of the random distributions of topics in the documents.

A predictive construction that allows for document-specific clustering into topics and shared topics across documents is given in \cite{Teh2006} as a hierarchical Chinese Restaurant process, or {\em Chinese franchise}, which is reminiscent of the hierarchical Hoppe's urn proposed for infinite hidden Markov models by \citep{Beal2002}, as we here describe. To each document $j$, let us associate a Hoppe's urn ${\mathcal U}_j$, that initially only includes $\alpha_j > 0$ black balls, then sample from each urn as described in Example \ref{ex:polyasequence}; however, whenever a new color is needed, pick it from an \lq \lq oracle urn" which is another Hoppe's urn, with initial number $\gamma$ of black balls and color distribution $G_0$, for simplicity assumed to be diffuse. The draws from the oracle Hoppe's urn represent the labels of the topics available for all documents; when colored, they are an exchangeable sequence with directing random measure $\tG \sim DP(\gamma, G_0)$. 
Conditionally on all the draws from the oracle urn (thus on $\tG$), 
the colored drawings from the document-specific Hoppe's urns $\mathcal{U}_j$ are independent exchangeable sequences $(\ttheta_{n,j})_{n \geq 1}$, with 
\begin{align*} 
\tilde\theta_{n,j} \mid \tG_j & \iidsim \tG_j \,  \\
\tG_j \mid \tG & \sim DP(\alpha_j, \tG) \, , 
\end{align*}
independently across $j$, and 
in turn, $\tG \sim DP(\gamma, G_0)$. This defines a  Hierarchical Dirichlet Process prior for $(\tG_1, \ldots, \tG_M)$, with parameters $(\alpha_1, \ldots, \alpha_M, \gamma, G_0)$.

This model is based on an
exchangeable structure at the latent stage, 
where (in line with the considerations in \cref{ex:pred-continuous}), one envisages an actual random partition. 
Differently from \cref{ex:polyasequence}, here the draws from the Hoppe's urns are latent variables,  since, at any time, an \lq \lq old" 
color could 
be picked 
from $\mathcal U_j$  or from the oracle urn. 
This leads to computational challenges, as we discuss in the next section.

Extensions of the hierarchical Dirichlet process 
include the hierarchical Pitman-Yor process \citep{teh2006-HPY}, 
and hierarchies of general discrete random measures leading to interesting combinatorial structures;  see Camerlenghi {\em et al.} \cite{camerlenghi2019-theory-hierarch}, and \cite{Catalano2023} and references therein.  
$\square$\end{example}

\subsection{Asymptotic partial exchangeability}\label{sec:asymptoticpe}
In the above example, and in fact more generally with partially exchangeable data, 
the predictive and the posterior distributions are not available in a \lq \lq closed'' 
(ideally, conjugate) analytic form - with a sometimes significant computational cost. To give some insight on the reasons for this difficulty, suppose for brevity that  we only have  two partially exchangeable sequences $(X_n)_{n\geq 1}, (Y_n)_{n\geq 1}$ 
and aim for an analytically tractable 
predictive distribution $\PP(X_{n+1} \in \cdot \mid x_{1:n},y_{1:n}, y_{n+1})$. 
Assuming $X_{n+1}$ $\indep Y_{n+1} \mid X_1,\dots,X_n,Y_1,\dots,Y_n$ would help, but typically  breaks partial exchangeability, except for trivial cases. On another extreme, a functionally simple 
inclusion  of $Y_{n+1}$ in the expression of the predictive distribution above may create {\em direct} dependence between the two sequences, and 
rather give  
{\em interacting} stochastic processes (see e.g. \cite{AlettiCrimaldiGiglietti2023-interacting} and references therein).
In fact, in partially exchangeable constructions one typically identifies a conditional independence structure of the kind $X_{n+1} \indep Y_{n+1} \mid X_1,\dots,X_n,Y_1,\dots,Y_n, U$ where $U$ is a latent random variable; (in a nonparametric setting with discrete priors, $U$ is an appropriate feature 
of the random partition, see e.g. \cite{camerlenghi2019-theory-hierarch}). 
While this may allow approximation schemes, for example through Gibbs sampling (\cite{Teh2006}, \cite{Lijoi2014}, \cite{Camerlenghi2017-JMA}), 
integrating out the latent $U$ to obtain the predictive distribution $\PP(X_{n+1} \in \cdot \mid x_{1:n}, y_{1:n}, y_{n+1})$ is not, generally, analytically manageable.   

Although there has been a sensible effort to find “closed form” expressions for predictive distributions for partially exchangeable models, the above considerations highlight that it is not easy to have partial exchangeability {\em and} also an analytically tractable predictive rule. 
This raises interest for predictive structures that only preserve partial exchangeability asymptotically, but are computationally easier. 
\citep{fortini2018} have proposed the notion of 
{\em partially conditionally identically distributed} (partially c.i.d.) sequences, which is equivalent to partial exchangeability for stationary data and preserves main 
properties of partially exchangeable sequences. In particular, partially c.i.d. processes are asymptotically partially exchangeable. Natural extensions of reinforced stochastic processes turn out to be partially c.i.d. For example, consider a family of sequences $[X_{n,j}]_{n\geq 1,j=1,\dots,M}$ such that $\PP(X_{1,j}\in\cdot)= P_{0,j}(\cdot)$ and for any $n \geq 1$
\begin{multline*}
     \PP(X_{n+1,j}\in\cdot\mid
     (X_{k,i})_{k\leq n, i=1,\dots,M}) 
     =\frac{\alpha_{0,j}P_{0,j}(\cdot)+\sum_{k=1}^n W_{k,j}\delta_{X_{k,j}}(\cdot)}{\alpha_{0,j}+\sum_{k=1}^n W_{k,j}},
\end{multline*}
 where  the random weights $W_{m,j}$ are strictly positive r.v.'s
 and may be functions of the observed values of the other sequences. It is proved in \citep{fortini2018} that if,  
 conditionally on $(X_{m,i},W_{m,i})_{m\leq n,i\leq M}$, the future observations  $X_{n+1,1},\dots,X_{n+1,M}$ are mutually independent and $W_{n,j}$ is independent of $X_{n,j}$,  $j=1,\dots,M$,   
then the sequences $[X_{n,j}]_{n\geq 1,j=1,\dots,M}$ are 
partially c.i.d. 

\subsection{ Markov exchangeability} \label{sec:markov}
The representation theorem \ref{th:representationTh-exch}  for exchangeable sequences gives the conceptual justification of the Bayesian inferential setting for random sampling.
A natural question is if there is a symmetry notion and a de Finetti-like representation theorem that justify the Bayesian inferential setting for Markov chains. 
In this section we recall the notion of Markov exchangeability 
\citep{diaconisFreedman1980} and its predictive characterization \cite{fortini2017}, and review Diaconis and Freedman's representation theorem
and a different representation that relates Markov exchangeability to partial exchangeability in the sense of de Finetti. 
Many models, for instance state-space models for nonstationary time series, are based on Markov chains; thus, these results also give insights on predictive constructions for Bayesian learning with temporal data, beyond Markov chains. 
\medskip

Let $\mathbb X$ be a  finite or countable set that includes at least two points, and $(X_n)_{n\geq 0}$ be a sequence of r.v.'s taking values in $\mathbb X$, 
and with probability law $\mathbb P$. 
The process $(X_n)_{n\geq 0}$ is {\em partially exchangeable in the sense of Diaconis and Freedman}, or, following the terminology of \cite{zaman1984} and \cite{zabell1995}, {\em Markov exchangeable}, if its probability law is invariant under finite permutations that do not alter the number of transitions between any two states; 
more precisely, if 
$\PP(X_0=x_0, \dots,X_n=x_n)=\PP(X_0=x'_0,\dots,X_n=x'_n)$ whenever $(x_0,\dots,x_{n})$  and  $(x_0',\dots,x'_{n})$ have the same initial value (i.e. $x_0=x_0'$)
and exhibit the same number of transitions from state $i$ to state $j$, for every $i,j\in\mathbb X$.

Under a recurrence condition, Diaconis and Freedman prove a de Finetti-like representation theorem for  Markov exchangeable sequences. The process $(X_n)_{n \geq 0}$ 
is recurrent if the initial state $x_0$ is visited infinitely many times
with probability one. Let us also define, for any $i,j \in \mathbb X$, the
transition counts 
$T_{i,j}^{(n)}$ as the number of transitions from state $i$ to state $j$ in $(X_0,\dots,X_n)$, 
and the matrix of normalized transition counts as the matrix with elements 
$\hat T_{i,j}^{(n)}=T_{i,j}^{(n)}/\sum_{k\in \mathbb X} T_{i,k}^{(n)}$ if the sum is different from zero, and zero otherwise.
\begin{theorem}[Diaconis and Freedman \citep{diaconisFreedman1980}, Theorem 7 and Remark 25]
\label{th:representation-markov}
Suppose that the process $(X_n)_{n \geq 0}$, starting at $x_0$, is recurrent. If $(X_n)_{n \geq 0}$ is Markov exchangeable, then 
\begin{itemize}
\item[{\rm i)}] With probability one, the matrix of normalized transition counts converges (in the topology of coordinate convergence) to a random limit $\tilde Q$; 
\item[{\rm ii)}] conditionally on $\tilde{Q}$, the process $(X_n)_{n \geq 0}$ is a Markov chain with transition matrix $\tilde Q$. 
\end{itemize}
\end{theorem}
In applications in Bayesian statistics, the probability law of the random limit $\tilde Q$, which is uniquely determined by $\PP$,  plays the role of the prior.

The following result gives a predictive characterization of Markov exchangeable processes;  it parallels Theorems \ref{th:exch-by-pred} and \ref{th:partial-by-pred}. 
\begin{theorem}[\citep{fortini2017}] \label{th:markov-by-pred}
A predictive rule $(P_n)_{n \geq 0}$ for a process $(X_n)_{n\geq 0}$ with $X_0=x_0$ and $X_n \in \mathbb X$ characterizes a Markov exchangeable process 
if and only if for every $n\geq 0$ and every $(x_1,\dots,x_n)$ the following conditions hold:
\begin{itemize}
    \item[{\rm i)}]  For every $j\in\mathbb X$,
        $\PP(X_{n+1}=j\mid 
        x_{0:n})
           $
   depends on $x_{0:n}$ only through $x_0$ and its transition counts;
       \item[{\rm ii)}] For every $k\geq 1$ and all strings $\bf{y}$, ${\bf y}'$ and ${\bf z}$ of elements in $\mathbb X$ that do not contain $x_n$ and have no common elements, the function that maps $(\bf y,y')$ into
    $
        \PP((X_{n+1},\dots,X_{n+k})=({\bf y,z},x_n,{\bf y',z},x_n)\mid x_{0:n})
            $
    is symmetric in ${\bf( y,y')}$. 
    \end{itemize}
\end{theorem}
This is proved in \cite{fortini2017}, where a predictive condition for recurrence is also given. 
The characterization becomes much simpler when the predictive distribution of $X_{n+1}$ only depends on the last visited state $x_n$ and on the $x_n$th row $\mathbf t_{x_n}$ of the matrix of transition counts: $\PP(X_{n+1}\!=\!y\mid \!x_{0:n})\!=\!p(y\!\mid\! x_n,\mathbf t_{x_n})$.
In this case, $(X_n)_{n\geq 0}$ is Markov exchangeable if and only if 
\begin{align}    \label{eq:markov} 
 p(y\mid\!x, \mathbf t\!)   
 p(z\!\mid \!x,\mathbf t+\mathbf e_y\!)=p(z\mid\! x,\mathbf t\!) p(y\mid\! x,\mathbf t+\mathbf e_z\!)
\end{align}
for every $\mathbf t$ and every $x,y,z\in\mathbb X$, where $\mathbf e_y$ and $\mathbf e_z$ have a $1$ at positions $y$ and $z$, respectively, and $0$ elsewhere. 
\begin{example}[{\em Reinforced urn scheme}]
\label{ex:reinforcedUrn-markov finito} Let $\mathbb X$ be finite or countable, and let $(X_n)_{n\geq 0}$  satisfy $X_0=x_0$ and 
          \begin{equation} 
   \label{eq:reinforced}
\PP(X_{n+1}=y\mid x_{0:n})=\frac{\alpha_{x_n} q_{x_n}(y)+ t_{x_n,y}} {\alpha_{x_n}+\sum_{j\in\mathbb X} t_{x_n,j}},
\end{equation}
where for every $x$, $\alpha_x$ is a positive number and $q_x(\cdot)$ is a probability mass function on $\mathbb X$.  
It is easy to verify that the predictive rule \eqref{eq:reinforced} satisfies \eqref{eq:markov}.
Moreover, by the L\'evy extension of the Borel-Cantelli lemma, the state $x_0$ is visited infinitely many times (see \cite{fortini2017} for the details). Hence, $(X_n)_{n\geq 0}$ 
is a mixture of Markov chains. 

For a finite state space, 
Zabell \cite{zabell1995} derived 
the predictive rule \eqref{eq:reinforced} from Johnson’s sufficiency postulate and 
assuming that $(X_n)_{n\geq 0}$ is recurrent and Markov exchangeable, characterizing independent 
 Dirichlet prior distributions on the rows of the random transition matrix.
$\square$ 
\end{example}

By the result i) in Theorem \ref{th:representation-markov}, the random transition matrix $\tilde Q$ has an empirical meaning as the limit of the matrix of normalized transition counts. It is of interest to know whether it also has an interpretation in terms of prediction, in the spirit of Proposition \ref{prop:predconv}.
We can show that this is possible 
by leveraging on an alternative characterization of Markov exchangeable processes, hinted in \cite{deFinetti1959} and \cite{zabell1995} and developed in \cite{fortini2002}, in terms of partial exchangeability of the matrix of successor states. 

The $n$th successor state of a state $x$ is defined as the state visited by the process  $(X_n)_{n\geq 0}$ just after the $n$th visit to state $x$. Denoting by  $\tau_n(x)$  the time of the $n$th visit to state $x$, with $\tau_n(x)=\infty$ if $x$ is not visited $n$ times, we can define the $n$th successor state of $x$ as
$$
S_{x,n}=X_{\tau_n(x)+1}
$$
if $\tau_n(x)$ is finite.  Let us collect the successor states for all $x$ in an array $[S_{x,n}]_{x \in \mathbb X, n \geq 1}$. 
Note that the $x$th row of $[S_{x,n}]$ has infinite length if the state $x$ is visited infinitely many times, otherwise it is  of finite length. The set of states that are visited infinitely many times depends on the path $\omega$, so does the length of the row $(S_{x,n})_n$. To avoid rows of finite length, \cite{fortini2002} enlarge the state space, by adding an external point $\partial$,
and define $S_{x,n}(\omega)=\partial$ if $\tau_n(x)(\omega)=\infty$. 

It is proved in \cite{fortini2002} that $(X_n)_{n\geq 0}$ 
 is a mixture of recurrent Markov chains if and only if the 
 array of successor states $[S_{x,n}]_{x\in\mathbb X, n\geq 1}$ 
is partially exchangeable by rows in the sense of de Finetti,   
(see Theorem 1 in \cite{fortini2002} for more details and for the extension to uncountable state spaces). 
 This  allows us to use the results in Section \ref{sec:partial} for the array $[S_{x,n}]_{x\in\mathbb X,n\geq 1}$ of successor states. For each $x$, the $x$th row $(S_{x,n})_{n \geq 1}$ is exchangeable, thus the successors $S_{x,n}$ of state $x$ are conditionally i.i.d. given a random probability mass function $\tilde{Q}_x$ on $\mathbb X$. 
 The probability masses $\tilde Q_{x,i} \equiv \tilde{Q}_x(i)$ are the limits of the empirical frequencies $\sum_{k=1}^n \delta_{S_{x,k}}(i)/n$, $i \in \mathbb X$, and correspond to the $x$th row of the random transition matrix $\tilde{Q}$. Partial exchangeability also implies that the rows of the array of successor states are not independent sequences: 
probabilistic dependence across them is introduced through the joint prior law of the vector $(\tilde{Q}_x, x \in \mathbb X)$, i.e. 
of the rows of the random transition matrix $\tilde{Q}$.  

 Moreover, the random transition  matrix is the limit of the predictive distributions, in the sense that, for all $x,i$, 
 \begin{align*}
& \lim_n \PP(S_{x, n+1}=i \mid S_{x, 1}, \ldots, S_{x, n}, V) \\
& = \lim_n \PP(S_{x, n+1}=i \mid S_{x, 1}, \ldots, S_{x, n}) 
 = \tilde{Q}_{x, i}
\end{align*}
where $V$ collects all the rows of the matrix of successors states but the $x$th.
This result refers to the successor states. In terms of the sequence $(X_n)_{n\geq 0}$, see Theorem 1 in \cite{fortini2017}. 

\medskip
Stochastic processes with reinforcement are again powerful tools 
in predictive constructions of Markov exchangeable sequences. An elegant construction, through  {\em edge reinforced random walks}  on a graph, is Diaconis and Rolles'  \citep{diaconisRolles2006} characterization of a conjugate prior for the transition matrix of a reversible Markov chain.  Developments for variable order reversible Markov chains are in \cite{bacallado2011}. 
The reinforced urn schemes in the following examples could also be read in terms of reinforced random walk on a graph (by  associating urns to the vertices).

\begin{example}[{\em Reinforced Hoppe urn processes}] \label{ex:reinforced-Hoppe}
The predictive rule of Example \ref{ex:reinforcedUrn-markov finito} was obtained by \cite{fortini2012hierarchical} through a class of \lq reinforced Hoppe urn processes', that 
includes other constructions in the literature as special cases.  
Let the sample space (or color space) be finite or countable. 
To each $x \in \mathbb{X}$, associate a Hoppe urn $\mathcal{U}_x$, with $\alpha_x$ black balls and discrete  color distribution $q_{x}$ on $\mathbb{X}$. Balls are extracted from each urn by Hoppe sampling as in \cref{ex:polyasequence}, but we now move across urns as follows. Pick $x_0$ from an initial distribution $q$ on $\mathbb{X}$,  set $X_0=x_0$, 
go to urn $U_{x_0}$ and pick a ball from it. Since the ball will be black, 
a color $x_1$ is sampled from $q_{x_0}$ and a ball of color $x_1$ is added in the urn, together with the black ball. Set $X_1=x_1$ and move to Hoppe urn $U_{x_1}$, and proceed simlarly. Let $(X_n)_{n \geq 0}$ be the process so obtained. In this construction, the draws from the state-specific Hoppe urns $\mathcal{U}_x$ represent the successors of state $x$ and are \Polya sequences, independent across $x$; thus, the process $(X_n)_{n \geq 0}$ is Markov exchangeable. Under mild conditions 
it is also recurrent (see \citep{fortini2017} for details).
It follows that a recurrent reinforced Hoppe urn process is conditionally Markov, and the prior on the random transition matrix $\tilde Q$ is such that the rows of $\tilde Q$, regarded as random measures on the state space $\mathbb X$,  are independent, with  
$\tilde Q_x \sim$ DP$(\alpha_x ,q_{x})$ 
(or Dirichlet distributions in the case of a finite state space).

As a special case, with a finite state space $\mathbb X$, suppose that, for each $x \in \mathbb X$, the color distribution $q_x$ has finite support in $\mathbb X$;  then the process $(X_n)_{n\geq 0}$  reduces to the {\em reinforced urn process} by 
\cite{muliereSecchiWalker2000}.  

If $\mathbb X=\{0,1,2, \ldots\}$ and for each $x \in \mathbb X$, the color distribution $q_{x}$
of urn $\mathcal U_x$ has positive masses only on $x+1$ and $x_0=0$, the process $(X_n)$ corresponds to 
the reinforced urn process proposed by \cite{walkerMuliere1997-betaStacy} for Bayesian 
survival analysis. Indeed, for this case,  the exchangeable sequence of the lengths of the $x_0$ blocks characterizes a novel {\em Beta-Stacy} prior can be used as a conjugate prior with exchangeable censored data. A version of this predictive construction also allows a generalization of the finite population Bayesian bootstrap \citep{lo1988},  
to include censored observations \citep{mulierewalker1998}.
$\square$
\end{example}   

A hierarchical version of the reinforced Hoppe urn process gives the popular {\em infinite hidden Markov model} proposed by Beal, Ghahramani and Rasmussen 
\citep{Beal2002} for Bayesian learning in hidden Markov models with an {\em unbounded} number of states.  

\begin{example}[{\em infinite Hidden Markov Model}] \label{eq:iHMM}
Suppose that the state space $(\theta_1^*, \theta_2^*, \ldots)$ 
is countable and {\em unknown}.
In \cite{Beal2002}, this is the state space of the {\em latent state process} $(X_n)_{n \geq 0}$ 
of a hidden Markov model where a new state may be added as the need occurs. 
The authors construct 
$(X_n)_{n \geq 0}$ through a predictive scheme that 
again envisages a reinforced Hoppe urn process;  but,  
differently from \cref{ex:reinforced-Hoppe}, and also from the construction in \cref{ex:HDP}, here Hoppe's urns are created as a new state (color) is discovered; and colors are drawn when the need occurs from a common \lq oracle' Hoppe urn with an initial number $\gamma$ of black balls and diffuse color distribution $P_0$. 
The process starts by picking a ball from the oracle urn; since the ball will be black, a first color, say $\theta_1^*$, 
is picked from $P_0$; and the black ball and an additional ball of color $\theta_1^*$ are returned in the oracle urn.  Then one sets $X_0=\theta_1^*$ and creates a Hoppe urn $\mathcal{U}_{\theta_1^*}$ with $\alpha$ black balls, picks a ball from it, and proceeds similarly. 
This generates a  process $(X_n)_{n \geq 0}$ that is recurrent and Markov exchangeable; thus, there exist $\tilde Q$ conditionally on which 
$(X_n)_{n \geq 0}$ 
is a Markov chain with transition matrix $\tilde Q$, and the construction characterizes the prior law on $\tilde Q$. 
The draws from the oracle urn generate the states of the process, and are a \Polya sequence with directing random measure  $\tilde P \sim$ DP$(\gamma, P_0)$. 
Conditionally on all the draws $(\theta_1^*, \theta_2^*, \ldots)$ from the oracle urns, thus on $\tilde P=P$, the process $(X_n)_{n \geq 0}$ is a reinforced Hoppe urn process as in \cref{ex:reinforced-Hoppe}, with state space $(\theta_1^*, \theta_2^*, \ldots)$; thus, the rows of $\tilde Q$, regarded as random distributions on the state space $(\theta_1^*, \theta_2^*, \ldots)$, have independent DP$(\alpha, P)$ distributions. 
Therefore, the prior law on the rows of $\tilde Q$, regarded as random distributions,  
is a hierarchical Dirichlet process 
with parameters $(\alpha, \gamma, P_0)$. Here, the construction of the prior is purely predictive; the hierarchical Dirichlet process was introduced later \cite{Teh2006}.

The predictive distribution of $X_{n+1}$ given $x_{0:n}$ is analytically complex; 
however, the predictive construction above can be exploited to design computational methods,  see e.g. \citep{Gael_Ghahramani_2011}. 
$\square$
\end{example}

\subsection{Row-column exchangeability} \label{sec:RCE}

Many data are in the form of arrays, graphs, matrices,  and forms of partial exchangeability are developed for general random structures.  In this section we 
 briefly review Aldous' notion of {\em row-column exchangeability}, or partial exchangeability for random arrays, and  refer to Aldous \citep{aldous1985} and Kallenberg \citep{kallenberg2005} for extensive treatment. An excellent review paper that also includes Bayesian models for exchangeable random structures in statistics and machine learning is \cite{orbanzRoy2015}. 
We do not even try to review the wide and growing literature on row-column exchangeable 
arrays,  and related theory of exchangeable random graphs and more recent theory for sparse graphs. We just recall basic concepts and the analogue of de Finetti representation theorem for row-column exchangeable arrays. In the predictive perspective of this paper, it would be interesting to include 
basic  properties of the predictive distributions, in analogy to what we have in 
Proposition \ref{prop:predconv} for exchangeable sequences. However, to the best of our knowledge, results of this nature for row-column exchangeable arrays are lacking. 
A first result,
that once more relates the problem with de Finetti's concept of partial exchangeability and holds for a fairly general class of row-column exchangeable arrays, is given in unpublished work  by \cite{augusto}. 

\medskip

In studying exchangeability, we have regarded the data $(X_1, \ldots, X_n)$ 
as elements of an infinite sequence $(X_n)_{n \geq 1}$. 
Similarly, here we consider an observed finite array $[X_{i,j}]_{i,j=1, \ldots, n}$ as a sub-array of an infinite random array $X = [X_{i,j}]_{i,j \geq 1}$; the $X_{i,j}$ are $\mathbb X$-valued random variables, where 
$\mathbb X$ is a Polish space.

\begin{definition}
An infinite random  
array $X=[X_{i,j}]_{i,j \geq 1}$ is {\em separately exchangeable} if
\begin{equation}\label{eq:RCE}
X \overset{d}= [X_{\sigma_1(i), \sigma_2(j)}]_{i,j \geq 1}
\end{equation} 
for all finite permutations $\sigma_1, \sigma_2$ of $\mathbb N$. It is {\em jointly exchangeable} if the above holds in the special case $\sigma_1=\sigma_2$. 
\end{definition}
Condition  (\ref{eq:RCE}) is equivalent to requiring that the rows of $X$ are exchangeable and the columns are exchangeable; it is thus referred to as {\em row-and-column exchangeability} (RCE). We will use the terminolgy {\em RCE array}
to mean that the array is either separately or jointly exchangeable. 
Separate exchangeability is an appropriate assumption if rows and columns of the array correspond with two distinct sets of entities; for example, rows correspond to users and columns to movies. If there is a single set of entities, for example the vertices of a graph, one may require invariance under permutations of the entities, that is, joint exchangeability.

Indeed, binary jointly exchangeable arrays give a representation of exchangeable random graphs.  
A random infinite graph (with known vertices, labeled by $\mathbb N$, and random edges)
is exchangeable if its probability law is invariant under every finite permutation of its vertices. Equivalently, if and only if  the corresponding adjacency matrix $X=[X_{i,j}]_{i,j \geq 1}$, where $X_{i,j}$ is the indicator of there being an edge $(i,j)$ in the graph, is jointly exchangeable.  
Actually, theoretical results for RCE  
arrays have been rediscovered in the developments of 
the limiting theory for large graphs initiated by Lov\'asz and Szegedy \cite{LovaszSzegedy2006}. The connection between graph limits and RCE 
arrays is given by Diaconis and Janson \citep{diaconisJanson}. We refer to the monograph by Lov\'asz \cite{Lovasz2013-book} for the graph limit theory.

Proving a de Finetti-like representation for RCE arrays has been more delicate than  expected.
The  representation theorem was independently given by Hoover \cite{Hoover1979} and Aldous \cite{aldous1981} and developed more systematically by Kallenberg, 
culminating in his 2005 monograph \cite{kallenberg2005}. 
The proof that appears in Aldous (\cite{aldous1981}; see also \cite{aldous1985}, Theorem 14.11) 
uses the concept of 
\lq coding'. 
The way this is used 
may be unfamiliar for some readers; to introduce it, note that de Finetti's representation theorem can be given (e.g. \cite{aldous1985}, page 129) as follows.  A sequence of r.v.'s  $(X_n)_{n\geq 1}$ is exchangeable if and only if there exists a measurable 
function $H : [0,1]^2 \rightarrow \mathbb X$ such that $(X_n)_{n\geq 1}$ can be {\em coded}  through i.i.d. uniform r.v.'s  $U, U_i, i \geq 1$ with a {\em representing function} $H$, 
that is,   $(X_n)_{n\geq 1} \overset{d}=( H(U_n, U))_{n\geq 1}$.
For example, binary r.v.'s $(X_n)_{n\geq 1}$ are exchangeable if and only if they can be generated by first picking $\theta$ 
from a prior law $\pi$ (through $\theta=\pi^{-1}(U)$, where $\pi^{-1}$ is the generalized inverse of the prior distribution function $\pi$), then 
sampling  
$X_i \iidsim$ Bernoulli$(\theta)$ (through $X_i=\mathbf 1_{(U_i\leq\theta)}$).

\begin{theorem}[Aldous-Hoover representation theorem for separately exchangeable arrays] \label{th:aldousHoover}
An infinite random array $X=[X_{i,j}]_{i,j \geq 1}$ is separately exchangeable 
if and only if there exists $H : [0,1]^4 \rightarrow \mathbb X$
such that $X$  can be coded by i.i.d. Uniform$(0,1)$  independent r.v.'s 
$U; U_i, i \geq 1; V_j, j \geq 1, U_{i,j}, i,j \geq 1$, with representing function $H$, that is 
$$[X_{i,j}]_{i,j \geq 1} \overset{d}{=} [X_{i,j}^*]_{i,j \geq 1 },
 \mbox{where} 
\; X^*_{i,j} = H(U, U_i, V_j, U_{i,j}). $$
\end{theorem}
The natural statistical interpretation is that $X_{i,j}$ is determined by a row effect $U_i$, a column effect $V_j$, an individual effect $U_{i,j}$ and an overall effect $U$.  

{\em Binary arrays}. To simplify, let us consider binary arrays. The representation theorem can be rephrased by  
saying that an infinite binary random array $X$ is separately exchangeable if and only if there exists a probability measure $\pi$ on the space of (measurable) functions from $[0,1]^2 \rightarrow [0,1]$ such that $X$ can be generated as follows (the r.v.'s $U_i, V_i, U_{i,j}$ are as in the theorem). Each row $i$ 
is assigned a latent feature $U_i$ and each column $j$ is assigned a feature $V_j$.  
Independently generate a function $W(\cdot, \cdot)$ from the probability distribution $\pi$ 
(through the uniform r.v. $U$). 
Given the features assignment  and $W$,  set $X_{i, j}=1$ with probability $W(U_i, V_j)$ (that is, $X_{i,j}=1$ if $U_{i,j} \leq W(U_i, V_j)$).
 Note that if  $W$ is fixed (not picked from $\pi$), the resulting array $[X_{i,j}]_{i,j \geq 1}$ is separately exchangeable by the symmetry of the construction. Denote by $P_W$ 
its probability law. Aldous-Hoover representation theorem proves that any separately exchangeable binary array is a mixture of such arrays. 

\begin{theorem}[Aldous-Hoover; binary arrays] \label{th:aldousHoover-binary}
Let $X=[X_{i,j}]_{i, j \geq 1}$ be an infinite separately 
exchangeable binary random array. Then, there is a probability distribution $\pi$ such that 
\begin{equation}    \label{eq:aldous-jointPW}
    \PP( X \in \cdot) = \int P_W( \cdot) d\pi(W). 
\end{equation}
\end{theorem}
Borrowing from the language of random graphs, a (measurable) map $W:[0,1]^2 \rightarrow [0,1]$ is called a {\em graphon}. A graphon defines a probability law $P_W$ as above, however this parametrization is not 
unique; in other words, 
in statistical sense, $W$ is not identifiable.
Indeed, if $W'$ is obtained from $W$ by a measure-preserving transformation 
of each variable, then clearly the associated process $[X'_{i,j}]_{i,j\geq 1}$ has the same joint distribution as $[X_{i,j}]_{i,j\geq 1}$. It has been proved that this is the only source of non-uniqueness \citep{kallenberg2005}. A unique parametrization can be obtained by substituting the graphons $W$ by equivalence classes. The results by Orbanz and Szegedy \cite{orbanzSzegedy2016} imply that this parametrization is measurable. See \cite{diaconisJanson} and \cite{orbanzRoy2015}  for a more extensive treatement. 

\medskip   
For {\em jointly} exchangeable arrays, there is an analogous representation result as Theorem \ref{th:aldousHoover},  with  
$X^*_{i,j} = H(U, U_i, U_j,U_{\{i,j\}})$, where $(U_i)_{i \geq 1}$ and $[ U_{\{i,j\}}]_{i,j\geq 1}$ are, respectively, a sequence and an array of independent uniform r.v.'s and $H$ is symmetric in $(U_i,U_j)$; see \citep{aldous1985}, Theorem 14.21. 
Note that the indexes of the $U_{\{i,j\}}$ are unordered and 
the array $[U_{\{i,j\}}]$ may be thought of as an upper-triangular matrix with i.i.d. uniform entries.

We give a version of the representation theorem for {\em binary} jointly exchangeable arrays; in particular, this applies to binary arrays representing the adjacency matrix of an infinite simple graph (undirected and with no multiple edges and self-loops). In this latter case, the adjacency matrix $[X_{i,j}]_{i,j\geq 1}$ is symmetric with a zero diagonal. 
A binary jointly exchangeable array can be constructed in a similar way as before, by now assigning features to vertices. Namely, each vertex $i \in \mathbb N$ is assigned a latent feature $U_i$, with $U_i \iidsim$ Uniform$(0,1)$; given the latent features and a graphon $W$, we  set $X_{i,j}=1$  with probability $W(U_i, U_j)$, independently for all $i,j$. The array $[X_{i,j}]_{i,j\geq 1}$ so constructed is {\em jointly} exchangeable by construction (and is symmetric if $W$ is such). Denote by $P^{(joint)}_W$ its probability law. The Aldous-Hoover representation theorem shows that any binary jointly exchangeable array can be constructed as a mixture of these $P^{(joint)}_W$. 
In other words, if $[X_{i,j}]_{i,j\geq 1}$ is an infinite jointly exchangeable binary array, then conditionally on the features and on the graphon, the $X_{i,j}$ are independent Bernoulli($\ttheta_{i,j}$) where $\ttheta_{i,j}=W(U_i, U_j)$. \\

As given, the Aldous-Hoover  theorem  
does not provide an empirical link for the elements of the representation. For exchangeable sequences, de Finetti's representation theorem is
complemented by a law of large numbers,  that gives an empirical meaning to the random directing measure $\tF$ as the limit of the sequence of empirical distributions; moreover, $\tF$ is also the limit of the 
predictive distributions $P_n$ (see Proposition \ref{prop:predconv}). 
For RCE  arrays, the notion of an empirical distribution and a law of large numbers  are given by Kallenberg \cite{kalleberg1999}, Theorem 3. 
The asymptotic theory is also thoroughly explained in  \citep{orbanzRoy2015}. 
Instead, no result seems available on convergence of predictive distributions, that relate the Aldous-Hoover representation to prediction. To our knowledge, 
a first result is given in \cite{augusto};  but we do not expand this further here. 

\section{Recursive algorithms and predictions}\label{sec:algorithms}

The predictive approach has been shown to be powerful in many contexts. 
A last but important (to us) point we want to make in this paper is that a Bayesian predictive approach can also be taken in less \lq classic' contexts, in particular to evaluate predictive algorithms, possibly arising from other fields, in order to obtain better awareness of their implicit assumptions and provide probabilistic quantification of uncertainty.  
We discuss this point 
for two recursive procedures. The first example is from \cite{fortini2020-newton}; the second one is new.

Recursive computations are particularly convenient in sequential learning from streaming data, where it is crucial to have predictions that can be quickly updated as new observations become available, at a constant computational cost and with limited storage of information; and  recursive procedures 
have been developed
since at least the work of Kalman \cite{kalman1960}.
Recent directions in a Bayesian predictive approach include, among others, 
\cite{hahnMartinWalker2018}, 
\cite{fong2021} and \cite{rigo2023-StatSinica}.  
In sections \ref{sec:newton} and \ref{sec:gradient} below,  
we examine two recursive algorithms for prediction with streaming data;  
and, in line with the principles at the basis of this paper, exposed in the Introduction,  
we show how they can be read as Bayesian predictive learning rules (although not exchangeable), unveiling the implied statistical model and obtaining Bayesian uncertainty quantification. 
In the examples, the implied model is asymptotically exchangeable, thus for $n$ large the \lq algorithm' provides a computationally simple approximation of an exchangeable Bayesian procedure. 

In some more detail, we can read the algorithms as particular cases of a broad class of Bayesian recursive predictive rules of the following form: $X_1 \sim P_0$ 
and for every $n\geq 1$ $X_{n+1} \mid X_1,\dots, X_n 
\sim P_n$, with
\begin{equation} \label{eq:recursivePred} 
\left\{
\begin{array}{l}
P_n = p_n(T_n), \\
T_n=h_n(T_{n-1},X_n),
\end{array}
\right.
\end{equation}
where  
$p_n$ and $h_n$ are given functions, and $T_n$ is a predictive sufficient summary 
(Sect. \ref{subsec:sufficiency}) of $X_1,\dots,X_n$. The form of $P_n$ allows storage of only the sufficient summaries and straightforward updating. Suitable specifications  lead to  desirable properties for the sequence $(X_n)_{n\geq 1}$ (in the examples, asymptotic exchangeability). 
 
In an exchangeable parametric setting, many common models, for example the Beta-Bernoulli scheme, have a recursive rule of the form \eqref{eq:recursivePred}. In a nonparametric setting, this holds 
for \Polya sequences, whose predictive rule (\ref{eq:predDP}) can be written recursively as 
$$P_n = \frac{\alpha+n-1}{\alpha+n} P_{{n-1}} + \frac{1}{\alpha+n} \delta_{{X_n}}. $$
In this case $T_n\equiv P_n$,  and the recursive rule applies directly to the predictive distributions. This extends to other discrete nonparametric schemes; it is however more de\-li\-ca\-te in the continuous case. 
Here, a class of sequences $(X_n)_{n \geq 1}$ that satisfy (\ref{eq:recursivePred}) are 
{\em measure-valued  \Polya sequences} (MVPS; \citep{SarievSavov2024}), characterized by 
$$
P_n(\cdot)= \frac{\gamma P_0(\cdot) + \sum_{i=1}^n R_{X_i}(\cdot)}{\gamma + \sum_{i=1}^n R_{X_i}(\mathbb{X})}
$$
where $R$ is a non-null finite 
transition kernel on the sample space $\mathbb X$ and $\gamma$ is a positive constant.
Letting $\mu_0(\cdot)=\gamma P_0(\cdot)$, we can write the predictive distributions, for $n\geq 1$,  as in \eqref{eq:recursivePred}, with 
$$
\left\{
\begin{array}{l}
P_n(\cdot)=\dfrac{\mu_n(\cdot)}{\mu_n(\mathbb X)}\\
\mu_n(\cdot)=\mu_{n-1}(\cdot)+R_{X_n}(\cdot).
\end{array}
\right.
$$
The predictive sufficient statistic $T_n$ in \eqref{eq:recursivePred} is, in this case, the random measure
$\mu_n$, which is updated 
by simply adding the random measure $R_{X_n}$ to $\mu_{n-1}$. This scheme 
extends the Hoppe's urn characterization of \Polya sequences shown in \cref{ex:polyasequence}:  
any set of colors $B$ in $\mathbb X$ has initially mass  $\mu_0(B)$;
then, at each step $n$, the mass of $B$   
is reinforced with a 
mass $R_{x_n}(B)$.
As proved in \cite{SarievSavov2024}, a measure-valued \Polya sequence $(X_n)_{n\geq 1}$ is exchangeable if and only if it coincides with a kernel-based Dirichlet sequence; unfortunately, as seen in \cref{ex:pred-continuous}, the latter seems quite limited for statistical applications; and so are {\em exchangeable} specifications of MVPS.
Moreover, natural extensions, for example allowing for random reinforcement (see e.g. \cite{fortiniPetSariev2021-regazzini}, \cite{sarievFortiniPetrone2023-dominant}) also require to go beyond  exchangeability. Indeed, MVPS can be  {\em asymptotically} exchangeable. The procedure we consider in the next section \ref{sec:newton} will be shown to be an asymptotically exchangeable generalized MVPS.

{\em Remark.} 
As seen in \cref{sec:exch},  asymptotic exchangeability holds for c.i.d. sequences. 
A class of recursive predictive rules that, 
under mild assumptions,  meets the c.i.d. condition, is presented 
in \cite{rigo2023}, (Sect. 4.1, Eqn (5)). 
Although this class is rather general, not all the predictive rules of the form (\ref{eq:recursivePred}) - in particular, not those arising in the following sections - are included in it. We need the generality and the predictive features of the class \eqref{eq:recursivePred}.

\subsection{Newton's algorithm  and recursive prediction in mixture models }
\label{sec:newton}
Michael Newton and collaborators (\cite{newton1998}, \cite{newtonZhang1999}, \cite{newton2002}) proposed a recursive procedure for  unsupervised sequential learning in mixture models, that extends an earlier proposal by Smith and Makov \citep{smithMakov1978}
and is referred to as the  {\em Newton's algorithm} in the Bayesian nonparametric literature. 
Let $(X_n)_{n\geq 1}$ 
be a sequence of r.v.'s taking values in ${\mathbb X} \subseteq {\mathbb R}^d$, 
and consider a mixture model 
$$
X_i\mid  G \stackrel{i.i.d.}{\sim}f_{G}(x)\equiv \int k(x\mid\theta)d G(\theta),
$$
where $k(x\mid\theta)$, $\theta\in\Theta\subseteq \mathbb R^k$, is a kernel density of known parametric form,
and $G$ is the unknown mixing distribution. Let us assume that the 
mixture model is 
identifiable. The Newton's algorithm estimates $G$ 
starting from an initial guess $G_0$ and 
recursively updating the estimate, as $x_1, x_2, \ldots$ become available, as 
\begin{equation} \label{eq:newton}
G_n(\cdot) = (1-\alpha_n) G_{n-1}(\cdot) + \alpha_n G_{n-1}(\cdot\mid x_n), 
\end{equation} 
where $\alpha_n$ and $G_{n-1}(\cdot \mid x_n)$ are as described in \cref{ex:pred-continuous} of Section \ref{sec:methodsConstruction}, and a simple choice for $\alpha_n$ is $\alpha_n=1/(\alpha+n)$ for some $\alpha>0$.
At step $n$, the algorithm returns $G_n(\cdot)$ as the estimate of $G(\cdot)$. 

This recursive procedure was suggested as a simple and computationally fast approximation of the intractable Bayesian solution in a Dirichlet process mixture model. In the latter,  
the mixing distribution is random and is assigned a DP$(\alpha G_0)$ prior; then the prior guess is 
$E(\tilde{G}(\cdot))=G_0(\cdot)$,  
and the first update, based on $x_1$, gives the Bayesian estimate 
$$ E(\tG(\cdot ) \mid x_1) = \frac{\alpha}{\alpha+1}  G_0(\cdot ) + \frac{1}{\alpha+1} G_0(\cdot\mid x_1). 
$$
Newton's algorithm (\ref{eq:newton}) replicates 
the same updating form for any $n > 1$. The resulting estimate $G_n$ 
deviates from the Bayesian solution $E(\tG(\cdot ) \mid x_{1:n})$,   
but is computationally much simpler; and, in practice, may give a surprisingly good approximation.
Based on $G_n$,
one can also obtain a plug-in estimate of the mixture density $f_G$, as  $f_{G_n}(x)=\int k(x \mid \theta) d G_n(\theta)$. Again, this differs from  
 the Bayesian density estimate $E(f_G(x) \mid x_{1:n})$ 
in the Dirichlet process mixture model, 
where $E(f_G(x) \mid x_{1:n})$ is also the predictive density of $X_{n+1}$ given $x_{1:n}$.   
Our point is thus that 
Newton's algorithm is 
using a different learning rule, namely the predictive density 
\begin{equation} \label{eq:pred-newton}
X_{n+1} \mid x_{1:n} \sim  f_{G_n}(x)=\int k(x \mid \theta) d G_n(\theta),
\end{equation} 
with $G_n$ as in \eqref{eq:newton}. 
This is of the form (\ref{eq:recursivePred}), with sufficient statistic $T_n=G_n$; and gives 
a generalized measure-valued \Polya sequence.  
Because, as seen in Section \ref{sect:exch-caratterization} and along this paper, the predictive rule characterises the probability law of the process $(X_n)_{n\geq 1}$, reading the algorithm as a {\em probabilistic} predictive rule allows to reveal the probability law 
that the researcher is implicitly 
assuming for the process. 
It is easy to see that \eqref{eq:pred-newton}  characterizes a probability law $\PP$ for $(X_n)_{n\geq 1}$ that is no longer exchangeable. 
However, one still has  
$X_{n+2} \mid x_{1:n} \overset{d}{=} X_{n+1} \mid x_{1:n}$. Thus,  the sequence $(X_n)_{n\geq 1}$ is c.i.d. (see section \ref{sec:asymptotic}), therefore asymptotically exchangeable. 
Actually, \cite{fortini2020-newton} prove stronger results: 
the asymptotic directing random measure of $(X_n)_{n\geq 1}$ has precisely density 
$f_{\tilde G}(x)= \int k(x \mid \theta) d \tG(\theta)$, where, $\PP$-a.s,  the random distribution $\tG$ is the
limit of the sequence $G_n$, and $f_{\tG}$  is the limit in $L^1$ of  the predictive density  $f_{G_n}(x)$. 

The above results imply that for $n\geq N$ large 
$$X_n \mid \tG \overset{i.i.d.}{\approx} f_{\tG}, 
$$
with a novel prior on the random mixing distribution 
$\tG$, 
that, interestingly, can 
select absolutely continuous distributions  
\citep{fortini2020-newton}; but is not known explicitly. However, one can sample from it, through the \lq sampling from the future' algorithm described in Section \ref{sec:predictive-based appr}. 
Moreover, in the same spirit as in Proposition \ref{prop:credible}, but referring to the mixing distribution,  
one can obtain an asymptotic Gaussian approximation of the posterior distribution  of $[\tilde G(t_1), \ldots, \tilde G(t_M)]$ given $x_{1:n}$.   
Our predictive methodology also allows to naturally obtain principled extensions, that otherwise would mostly be heuristic; see again \cite{fortini2020-newton}.

\subsection{Online gradient descent and prediction}\label{sec:gradient}

Consider the problem of classifying items as \lq type $0$' or \lq type $1$' based on a $d$-dimensional vector of features, for example through a neural network or a generalized linear model. Let the items arrive sequentially, and, for every $n\geq 1$, let $Y_n$ and $X_n$ represent the \lq type' and features of the $n$th item, respectively. Typically, the relationship between $X_n$ and $Y_n$ is modelled through $\PP(Y_n=1\mid x_n)=g(x_n,\beta)$,  where $g$ is a known function and $\beta$ is an unknown $d$-dimensional parameter, and the $X_i$ are assumed to be i.i.d. from a distribution (known or unknown) $P_X$. 
Given a sample, or \lq training set', $(x_i,y_i)_{i=1,\dots,n}$, an estimate of the parameter $\beta$ can be obtained by  minimizing, with respect to $\beta$, a loss function $L(\beta;x_1,y_1,\dots,x_n,y_n)$ measuring the difference between the actual values of $y_1,\dots,y_n$ and the ones predicted by the model.
While efficient algorithms exist to solve this optimization problem, the computational cost becomes substantial when $\beta$ is high-dimensional. Additionally, 
if data arrive sequentially, the process must be restarted from scratch with each new data point.
In this context,  $\beta$ can be estimated by an {\em online learning} \citep{smale2006} procedure, based on the stochastic approximation \citep{robbins1951} of the gradient descent dynamic: $\beta$ is initialised at time zero as $\beta_0$ (which can be random or deterministic), and then updated, at each new data $(x_n,y_n)$,  by \lq\lq moving'' it
along the direction that minimizes $L(\beta_{n-1};x_n,y_n)$, 
(that is opposite to the direction of the gradient with respect to $\beta_{n-1}$):
\begin{equation}\label{eq:gradientDescentgen}
\beta_n=\beta_{n-1}-\frac 1 n \nabla_{\beta} L(\beta_{n-1};x_n,y_n).
\end{equation}
Although the results of this section hold for other choices of $L$ (for example quadratic loss) and $g$, here we consider the typical case of binary cross entropy loss \\
$L(\beta;x_1,y_1,\dots,x_n,y_n) = 
-\sum_{i=1}^n [y_i\log_2(g(x_i,\beta))+(1-y_i)\log_2(1-g(x_i,\beta)) ]$
 and logistic function 
\begin{equation}\label{eq:logistic}
g(x,\beta)=\frac{e^{x^T\beta}}{1+e^{x^T\beta}}.
\end{equation}
In this case, \eqref{eq:gradientDescentgen} becomes
\begin{equation}\label{eq:gradientDescent}
\beta_n=\beta_{n-1}+\frac{1}{ n\log 2} (y_n-g(x_n,\beta_{n-1})) x_n.
\end{equation}

If $P_X$ is known, we can reinterpret the algorithm as a Bayesian predictive learning rule of the form 
\eqref{eq:recursivePred}, where
$T_n=\beta_n$ is updated at each new observation $(x_n,y_n)$ as in \eqref{eq:gradientDescent}, 
and we assume that, for $y=0,1$,
\begin{align}
\label{eq:gradient}
&\PP(X_{n+1}\in dx,Y_{n+1}=y|x_{1:n},y_{1:n})\\
&=g(x,\beta_n)^y(1-g(x,\beta_n))^{1-y}P_X(dx),\nonumber
\end{align}
with $g$ as in \eqref{eq:logistic}. In fact, the assumption that $P_X$ is known (or has been estimated separately) is only instrumental for the theoretical results; our final result does not require to know $P_X$.

This predictive rule is not consistent with exchangeability of the sequence $((X_n,Y_n))_{n\geq 1}$;  
however, under mild assumptions, exchangeability 
holds asymptotically, as shown in the following proposition. All the proofs are
in Section A5 
of the Supplement \citep{supplement}.
\begin{proposition}
\label{prop:gradient}
       Let $( (X_n, Y_n) )_{n \geq 1}$ have probability law $\PP$ characterized by the predictive rule  
       \eqref{eq:gradientDescent}-\eqref{eq:gradient}, where $g$ is given by \eqref{eq:logistic},
      $E(||\beta_0||^2)<\infty$, and 
           $P_X$ has bounded support.        
       Then:
       \begin{itemize}
          \item[{\rm i)}] The sequence of random vectors 
       $(\beta_n)_{n\geq 0}$ converges  $\PP$-a.s.
        to a random limit $\tilde \beta$ and, for every $n\geq 0$, $\beta_n=E(\tilde\beta\mid \beta_0,X_1,Y_1,\dots,X_n,Y_n)$;
    \item[{\rm ii)}]
   The sequence of random vectors  $((X_n,Y_n))_{n\geq 1}$ is $\tilde{P}_{X,Y}$-asymptotically exchangeable, with the random measure  
   $\tilde P_{X,Y}$ such that the conditional distribution $\tilde{P}_{Y|X=x}$ 
    of $\; Y$ given $X=x$ is Bernoulli($g(x,\tilde\beta$)). 
                 \end{itemize}            
    \end{proposition}
Informally, this implies that, for $n$ large, $$Y_n \mid \tilde\beta, x_n \overset{indep}{\approx}  \mbox{Bernoulli}(g(x_n,\tilde\beta)). $$
    
The posterior distribution of the random vector $\tilde\beta$ remains unknown. However, for $n$ large, 
it can be approximated by a multivariate Normal distribution centered in $\beta_n$.
\begin{proposition}\label{prop:gradient2}
     Under the assumptions of 
     Proposition \ref{prop:gradient}, with $P_X$ being non-degenerate on any linear subspace of $\mathbb R^d$,
     the conditional distribution of $\sqrt n(\tilde \beta-\beta_{n})$, given $\beta_0,X_1,Y_1,\dots,X_n,Y_n$, converges $\PP$-a.s., as $n \rightarrow \infty$,
     to a multivariate Normal distribution with mean zero  and random covariance matrix  
\begin{align}\label{eq:U}
   U=(\log 2)^{-2} \int \!\!x x^T g(x, \tilde \beta)(1\!-\!g(x, \tilde \beta))P_X(dx). 
\end{align}
\end{proposition}
The random matrix $U$, that depends on the unknown parameter $\tilde\beta$, can be approximated by replacing $\tilde \beta$ with $\beta_n$. Thus, for $n$ large
$ \tilde\beta \mid x_{1:n}, y_{1:n} \approx \Norm_{d}(\beta_n, U_n/n)$, 
with $U_n=(\log 2)^{-2}\int x x^T g(x,\beta_n)(1-g(x,\beta_n))dP_X(x)$. 

The following alternative approximation of $U$ does not require to know $P_X$.

\begin{proposition}\label{prop:gradient3}
   Under the assumptions of Proposition \ref{prop:gradient2}, as $n\rightarrow\infty$,
   \begin{itemize}
       \item[{\rm i)}] The statistic 
       $ V_n=\frac 1 n \sum_{k=1}^n k^2(\beta_k-\beta_{k-1})(\beta_k-\beta_{k-1})^T $
converges $\PP$-a.s. to the random matrix $U$ in \eqref{eq:U};
\item[{\rm ii)}]  The conditional distribution of $
\sqrt n V_n^{-1/2}( \tilde\beta-\beta_n)$, given $\beta_0,X_1,Y_1,\dots,X_n,Y_n$ converges $\PP$-a.s. to the standard multivariate Normal distribution.
   \end{itemize}
\end{proposition}
Thus, for $n$ large, for $\PP$-almost all sample paths,  
$$ \tilde\beta \mid x_{1:n}, y_{1:n} \approx \Norm_d(\beta_n, V_n/n), $$
which can be used, in particular, to provide asymptotic credible sets.

\vspace{2mm}
{\em Remark.} 
The proofs of Propositions \ref{prop:gradient}, \ref{prop:gradient2} and \ref{prop:gradient3} are based on a key martingale property of the sequence $(\beta_n)_{n\geq 0}$. These results can be generalized to other algorithms as long as the martingale property holds and certain moment bounds are met. Although our techniques are not directly applicable without the martingale property, 
extending the Bayesian interpretation beyond martingale-based learning appears feasible, since  many algorithms are based on stochastic approximations with well-understood limit theorems and convergence rates.  Also, computational strategies  
such as Approximate Bayesian Computation or Variational Bayes might be read as using a predictive learning rule whose properties could  be studied in our predictive approach. 

\section{Final remarks}\label{sec:conclusions}
We have offered a review, from foundations to some recent directions, of principles and methods for Bayesian predictive modeling; 
and of course a lot could not be covered. We barely mentioned that prediction is not, in fact, the ultimate goal, but the basis for decisions to be taken under risk. Also, the paper is, somehow unavoidably, mostly theoretical, aiming at discussing fundamental concepts; but a predictive approach involves our perspective 
in inference and in 
any statistical problem, with evident practical implications; ultimately, the basic principle is that, differently from inferential conclusions, predictions  can be checked with facts. 

\begin{acks}[Acknowledgments]
We thank the three reviewers for their valuable comments, and are grateful to Sara Wade for interesting suggestions.  
Both authors acknowledge funding by the European Union grant \lq Next Generation EU Funds, PRIN 2022 (2022CLTYP4)\rq. 
\end{acks}

\begin{supplement}
\stitle{Supplement}
\sdescription{The Supplement collects  the proofs for the results in the paper.}
\end{supplement}

\bibliographystyle{imsart-number} 
\bibliography{statsc}

\end{document}


\begin{frontmatter}

\title{Supplement to \lq\lq Exchangeability, prediction and predictive modeling in Bayesian statistics''}

\author{Sandra Fortini and Sonia Petrone}
\end{frontmatter}

\runtitle{Prediction and exchangeability}

\setcounter{equation}{0}
\setcounter{figure}{0}
\setcounter{table}{0}
\setcounter{page}{1}
\setcounter{section}{0}
\makeatletter
\renewcommand{\theequation}{A\arabic{equation}}
\renewcommand{\thesection}{A\arabic{section}}

\section{Preliminary results}\label{sec:proof1}
We first recall two results in the literature that will be used in the subsequent arguments. The first result \citep{crimaldi2009} is an {\em almost sure convergence of conditional distributions} result for martingales; the notation has been adapted to fit with this paper.

\begin{theorem}[Theorem 2.2 in \citep{crimaldi2009}]\label{th:crimaldi}
On $(\Omega,\mathcal F,\mathbb P)$, let $(M_n)_{n\geq 0}$ be a real martingale with respect to the filtration $\mathcal G=(\mathcal G_n)_{n\geq 0}$. Suppose that $(M_n)_{n\geq 0}$ converges in $L^1$ to a random variable $M$. Moreover, setting, for $n\geq 1$,
$$
U_{n}\equiv n\sum_{m\geq n}(M_m-M_{m-1})^2,\, Y\equiv \sup_n\sqrt n |M_n-M_{n-1}|,
$$
assume that the following conditions hold:
\begin{itemize}
\item[{\rm i)}] The random variable $Y$ is integrable.
\item[{\rm ii)}] The sequence $(U_{n})_{n\geq 1}$ converges $\PP$-a.s. to a positive random variable $U$.
\end{itemize}
Then,  for every $z\in\mathbb R$,
$$
\PP(\sqrt n (M_n-M)\leq z\mid \mathcal G_n)\stackrel{a.s.}{\rightarrow}\Phi(U^{-1/2}z),
$$
as $n\rightarrow\infty$, where $\Phi$ denotes the standard normal cumulative distribution function.
\end{theorem}
A careful inspection of the proof of Theorem 2.2 in \citep{crimaldi2009} shows that, if $U_{n}$ converges to a non-negative random variable $U$, then the thesis can be modified as follows: for $\omega$ in a set of probability one
\begin{align*}
\PP(\sqrt n (M_n-M)\leq x\mid \mathcal G_n)(\omega)
\rightarrow
\left\{
\begin{array}{ll}
\Phi(U^{-1/2}x)(\omega)&U(\omega)>0\\
\mathbf 1_{[0,+\infty))}(x)&U(\omega)=0.
\end{array}
\right.
\end{align*}

In applying Theorem \ref{th:crimaldi}, a critical point is proving the convergence of $(U_{n})_{n\geq 1}$. The following result provides sufficient conditions.
\begin{theorem}[Lemma 4.1 in \citep{crimaldi2016}]
\label{lemma:series}
Let $\mathcal G=(\mathcal G_n)_{n\geq 0}$ be a filtration and $(Z_n)_{n\geq 1}$ be a $\mathcal G$-adapted sequence of real random variables such that $E(Z_n|\mathcal G_{n-1})\rightarrow Z$, $\PP$-a.s. for some real random variable $Z$. Moreover, let $(a_n)_{n\geq 1}$ and $(b_n)_{n\geq 1}$ be two positive sequences of strictly positive real numbers such that
$$
b_n\uparrow+\infty,\quad\quad \sum_{n=1}^\infty \dfrac{E(Z_n^2)}{a_n^2b_n^2}<+\infty.
$$
Then we have:
\begin{itemize}
\item[{\rm i)}] If $ \frac{1}{b_n}\sum_{m=1}^n \frac{1}{a_m}\rightarrow\gamma$ for some constant $\gamma$, then 
$$\dfrac{1}{b_n}\sum_{m=1}^n\dfrac{Z_m}{a_m}\stackrel{a.s.}{\rightarrow}\gamma Z.$$
\item[{\rm ii)}] If $b_n\sum_{m\geq n}\dfrac{1}{a_mb_m^2}\rightarrow\gamma$ for some constant $\gamma$, then $$b_n\sum_{m\geq n}\dfrac{Z_m}{a_mb_m^2}\stackrel{a.s}{\rightarrow}\gamma Z.$$
\end{itemize}
\end{theorem}

\section{Proofs for Section 2
}\label{sec:proof2}
Before proving Proposition 2.6, 
we give a preliminary convergence result.
\begin{proposition} \label{prop:pred}
Let $(X_n)_{n \geq 1} \sim \PP$ 
be a c.i.d. sequence of real-valued r.v.'s, with predictive rule $(P_n)_{n \geq 0}$, and take ${\bf t}=(t_1,\dots,t_k)$  such that $\PP(X_1\in \{t_1,\dots,t_k\})=0$. 
If the following conditions hold:
\begin{align*}
    &E(\sup_n \sqrt n |\Delta_{t_i,n}|)<+\infty&i=1,\dots,k\\
    &\sum_{n=1}^\infty n^2 E(\Delta_{t_i,n}^4)<+\infty&i=1,\dots,k\\
    &E(n^2 {\bf \Delta}_{{\bf t},n}{\bf\Delta}_{{\bf t},n}^T\mid X_1,\dots,X_n)\rightarrow U_{\bf t}  &\PP\mbox{-a.s.}
    \end{align*}
for a positive definite random matrix 
$U_{\bf t}$, 
then,  $\PP$-a.s., 
\begin{equation*}
 \sqrt n \left[\begin{array}{c}\tilde F(t_1)-P_n(t_1) 
 \\ \dots \\\tilde F(t_k)-P_n(t_k) 
 \end{array}\right] \mid X_1,\dots,X_n
 \overset{d}{\rightarrow} \Norm_k(0, U_{\bf t}).
\end{equation*}
\end{proposition}
\begin{proof}
By Cram\'er-Wald device, it is sufficient to show that, for every vector $u=[u_1\dots u_k]^T$ with $||u||=1$, every $z\in \mathbb R$, $\PP$-a.s.
\begin{align*}
    \PP(\sqrt n \sum_{i=1}^ku_i(\tilde F(t_i)-P_n(t_i))\leq z\mid X_1,\dots,X_n)
\rightarrow \Phi((u^T U_{\bf t}u)^{-1/2}z),
\end{align*}
where $\Phi$ denotes the standard normal cumulative distribution function.
The proof is based on \cref{th:crimaldi} and \cref{lemma:series}. 

The sequence $(X_n)_{n\geq 1}$ is c.i.d., thus $(P_n)_{n\geq 0}$ is a martingale with respect to the natural filtration $(\mathcal G_n)_{n\geq 0}$ of $(X_n)_{n\geq 1}$. Moreover the asymptotic directing random measure $\tilde F$ of $(X_n)_{n\geq 1}$ is the $\PP$-a.s. limit of $P_n$ (in the topology of weak convergence) and satisfies, for every $n\geq 1$, $E(\tilde F\mid \mathcal G_n)=P_n$. In particular, $E(\tilde F)=P_0$, which implies that
$E(\tilde F\{t_1,\dots,t_k\})=P_0(\{t_1,\dots,t_k\})=0$. It follows that $\tilde F(\{t_1,\dots,t_k\})=0$ ($\PP$-a.s.). Hence the bounded martingale $(P_n(t_1),\dots,P_n(t_k))_{n\geq 0}$ converges $\PP$-a.s. and in $L^1$ to $(\tilde F(t_1),\dots,\tilde F(t_k))$. 

For every $n\geq 0$, let $M_n=\sum_{i=1}^k u_iP_n(t_i)$. Then, $(M_n)_{n\geq 0}$ is a martingale with respect to $(\mathcal G_n)_{n\geq 0}$, converging in $L^1$ to $\sum_{i=1}^k u_i\tilde F(t_i)$. Notice that, for every $n\geq 1$, $M_n-M_{n-1}=\sum_{i=1}^ku_i\Delta_{t_i,m}$. To prove that condition {\rm i)} of \cref{th:crimaldi} holds we can write that
\begin{align*}
    \sup_n\sqrt n\left|\sum_{i=1}^ku_i\Delta_{t_i,n}\right|
    \leq\sup_n\sqrt n\sum_{i=1}^k\left|\Delta_{t_i,n}\right|
    \leq\sum_{i=1}^k\sup_n\sqrt n\left|\Delta_{t_i,n}\right|.
\end{align*}
Thus,
\begin{align*}
    &E\left(\sup_n\sqrt n\left|\sum_{i=1}^ku_i\Delta_{t_i,n}\right|\right)\leq \sum_{i=1}^kE\left(\sup_n\sqrt n\left|\Delta_{t_i,n}\right|\right),
\end{align*}
which is finite by the first assumption of \cref{prop:pred}.
To verify condition {\rm ii)} of \cref{th:crimaldi}, we employ \cref{lemma:series}{\rm ii)}, with $$Z_n=n^2(M_n-M_{n-1})^2=n^2(\sum_{i=1}^ku_i\Delta_{t_i,n})^2,$$ $u_n=1$ and $b_n=n$. First,
\begin{align*}
    E(Z_n\mid \mathcal G_{n-1})&=E\left(n^2(\sum_{i=1}^k u_i\Delta_{t_i,n})^2\mid\mathcal G_{n-1}\right)\\
    &=\sum_{i,j=1}^k u_iu_jE\left(n^2\Delta_{t_i,n}\Delta_{t_j,n}\mid\mathcal G_{n-1}\right),
\end{align*}
which converges $\PP$-a.s. to $u^TU_{\bf t}u$ by the third assumption of \cref{prop:pred}.
Furthermore
\begin{align*}
   &\sum_{n=1}^\infty n^2E\left(\left(\sum_{i=1}^k u_i\Delta_{t_i,n}\right)^4\right)
\leq  \sum_{n=1}^\infty n^2E\left(\left( \sum_{i=1}^k \Delta_{t_i,n}^2\right)^2\right)\\
    &\quad\leq  k^2 \sum_{n=1}^\infty n^2E\left(\left( \frac 1 k \sum_{i=1}^k \Delta_{t_i,n}^2\right)^2\right)
   \leq  k^2 \sum_{n=1}^\infty n^2E\left(\frac 1 k \sum_{i=1}^k \Delta_{t_i,n}^4\right)\\
      &\quad\leq  k  \sum_{i=1}^k\sum_{n=1}^\infty n^2E\left( \Delta_{t_i,n}^4\right),
\end{align*}
which is finite by the second assumption of \cref{prop:pred}. It follows by \cref{lemma:series} that
\begin{align*}
    n\sum_{m\geq n}(M_m-M_{m-1})^2&=b_n\sum_{m\geq n}\frac{Z_m}{b_m^2}
\end{align*}
converges $\PP$-a.s. to the $\PP$-a.s. limit of $E(Z_n\mid \mathcal G_{n-1})$, that is to $u^TU_tu$. This proves that assumption ii) of \cref{th:crimaldi} holds. Therefore, for every $z\in\mathbb R$, $\PP$-a.s.,
\begin{align*}
   &\PP(\sqrt n \sum_{i=1}^ku_i(\tilde F(t_i)-P_n(t_i))\leq z\mid X_1,\dots,X_n)\\
   &=\PP(\sqrt n (M_n-M)\leq z\mid X_1,\dots,X_n)\\
   &\stackrel{a.s.}{\rightarrow}\Phi((u^TU_{\bf t}u)^{-1/2}z).
\end{align*}
\end{proof}

\begin{proof}[Proof of Proposition 2.6
]
By the properties of stable convergence and \cref{prop:pred}, it is sufficient to show that $V_{n,{\bf t}}$ is positive definite for $n$ large enough, and converges to $U_{\bf t}$, $\PP$-a.s, with respect to the operator norm $||\cdot ||_{op}$. Since $U_{\bf t}$ is positive definite, 
then $V_{n,{\bf t}}$ is positive definite for $n$ large enough if $(V_{n,{\bf t}})_{n\geq 1}$ converges to $U_{\bf t}$ in the operator norm. Thus, it is sufficient to show that $(V_{n,{\bf t}})_{n\geq 1}$ converges to $U_{\bf t}$ in the operator norm, or, equivalently that, for every vector $u=[u_1\dots u_k]^T$ with $||u||=1$, 
$$
u^T V_{n,{\bf t}}u
\rightarrow u^T U_{\mathbf t} u, 
$$
$\PP$-a.s. Notice that
\begin{align*}
    u^T V_{n,{\bf t}}u 
   & =\sum_{i,j=1}^k u_iu_j \frac{1}{n}\sum_{m=1}^n m^2\Delta_{t_i,m}\Delta_{t_j,m}
=\frac 1 n \sum_{m=1}^n m^2\left(\sum_{i=1}^k u_{t_i}\Delta_{t_i,m}\right)^2
\\&=\frac 1 n \sum_{m=1}^n Z_m,
\end{align*}
with $Z_1,Z_2,\dots$ defined as in the proof of \cref{prop:pred}. Since we are working under the assumptions of \cref{prop:pred}, we can employ all the findings in its proof. In particular, we know that 
$(E(Z_n\mid\mathcal G_{n-1}))_{n\geq 1}$ converges $\PP$-a.s to $u^T U_{\bf t}u $, 
and $\sum_{n=1}^\infty E(Z_n^2)/n^2<+\infty$. By \cref{lemma:series} {\rm i)} with $a_n=1$  and $b_n=n$, we obtain that
$
\frac{1}{n}\sum_{m=1}^n Z_m
$
converges $\PP$-a.s. to $u^T U_{\bf t} u$, which completes the proof.
\end{proof}

\section{Proofs for Section 3 
}\label{sec:proof3}

\begin{proof}[Proof of Theorem 3.3
]
Consider a countable class of convergence determining sets $A$ such that $P_0(\partial A)=0$. For every $A$ in the class, $\tilde P(\partial A)=0$, $\PP-a.s.$ Thus, 
\begin{equation}
\label{eq:convergence}
T(\hat F_n)\rightarrow  T(\tilde F), \quad q_n(A, T(\hat F_n))\rightarrow \tilde F(A)
\end{equation}
hold for every $A$ in the class, with probability one. Fix $A$.
Since $(q_n(A, t))_{n\geq 0}$ are continuous in $t$, uniformly with respect to $t$ and $n$, then for every $\epsilon>0$ there exists $\delta=\delta(A)$ such that $$|q_n(A,t)-q_n(A,t')|<\epsilon$$ for every $n$, whenever $t,t'\in\mathbb T$ satisfy $||t-t'||<\delta$. 
Let $\omega$ and $\omega'$ be such that \eqref{eq:convergence} holds for $\omega$ and $\omega'$, and $T(\tilde F)(\omega)=T(\tilde F)(\omega').$
Then, for $n$ large enough $||T(\hat F_n)(\omega)-
T(\hat F_n)(\omega')||<\delta$, which implies that, for $n$
sufficiently large $$|q_n(A, T(\hat F_n(\omega)))-
q_n(A, T(\hat F_n(\omega')))|<\epsilon.$$ In turn, this 
implies that $|\tilde F(A)(\omega)-\tilde F(A)(\omega')|
<\epsilon$. Since this is true for every $\epsilon$, then 
$\tilde F(A)(\omega)=\tilde F(A)(\omega')$. Since the class of sets $A$ is countable, then there exists a set $N\in \mathcal F$ with $\PP(N)=0$, such that $\tilde F(\cdot)(\omega)=\tilde F(\cdot)(\omega')$ for every $\omega,\omega'\in N^c$ satisfying 
$T(\tilde F)(\omega)=T(\tilde F)(\omega')$. It follows that there exists a function $F(\cdot\mid t)$  such that 
$\tilde F(\cdot)(\omega)=F(\cdot\mid T(\tilde F(\omega)))$ for every $\omega\in N^c$. 
Extending arbitrarily $F(\cdot \mid t)$ outside the set $T(\tilde F(N^c))$, we obtain $\tilde F(\cdot)=F(\cdot\mid T(\tilde F))$, $\PP$-a.s. The measurability of  $F$ can be proved by classical arguments.
\end{proof}

\begin{proof}[Proof of Proposition 3.7
] 
For any $k\geq 1$ and any partition $(A_1, \ldots, A_k)$,  we have from \cite{zabell1982} that there exists $(\alpha_{A_1}, \ldots, \alpha_{A_k})$ such that 
\begin{align*} \label{eq:suff-DP-partition}
    \PP(X_{n+1} \in A_j \mid x_{1:n}) = \PP(X_{n+1} \in A_j \mid n_{A_j}) 
    =\frac{ \alpha_{A_j} + n_{A_j}}{\alpha_{A_1, \ldots, A_k} + n}  ,
 \end{align*}
where, for each set $A$,  $n_A$ is the number of observations $x_1,\dots,x_n$ in $A$, $\alpha_{A_1, \ldots, A_k}=\sum_{i=1}^k \alpha_{A_i}$ and the first equality comes from (3.4). 
Moreover, $(\!\tF(\!A_1\!)\!, \ldots,\tF(\!A_k\!)\!)$ $\sim Dirichlet(\alpha_{A_1}, \ldots, \alpha_{A_k})$.  Any set $A$ can be interpreted as an element of a finite partition, and, by the properties of the Dirichlet distribution, $\alpha_{A}$ does not depend on the partition. In particular, $\alpha \equiv \alpha_{A_1,\dots,A_k}$ does not depend on $A_1,\dots,A_k$. 
By the assumption of exchangeability, we have that $E(\tF(A))= \PP(X_1 \in A)= P_0(A)$ for every $A$. On the other hand, by the properties of the Dirichlet distribution,  
$E(\tF(A))= \alpha_{A}/\alpha $, for any $A$.  Hence $\alpha_A=\alpha P_0(A)$, which implies that, for every partition $(A_1,\dots,A_k)$, $(\tF(A_1), \ldots, \tF(A_k)) \sim Dirichlet(\alpha  P_0(A_1), \dots, \alpha  P_0(A_k))$. 
Since all the finite dimensional distributions of $\tilde F$ coincide with the ones of a Dirichlet
process with parameters  $(\alpha, P_0)$, then  $\tilde F\sim DP(\alpha, P_0)$.
\end{proof}

\noindent{\em Remark.}
    The proof of Proposition 3.7 
    relies on the characterization of the Dirichlet distribution (i.e. a Dirichlet process on a finite set $\{1,\dots,k\}$)  by Zabell  \cite{zabell1982}.  Zabell assumes that $P_0(\{j\})>0$ for every $j=1,\dots,k$. However, if $P_0(\{j\})=0$ for a specific $j$, one can apply the result to the reduced space where $j$
    has been removed, demonstrating that 
$\tF$ restricted to this reduced space follows a Dirichlet distribution. The law of $\tF$ on the entire space $\{1,\dots,k\}$ can then be obtained by setting 
  $\tilde F(\{j\})=0$. This distribution is still referred to as Dirichlet distribution (see e.g. \cite{ferguson1973}).

\section{Proofs for Section 4 
}\label{sec:proof4}

\begin{proof}[Proof of Theorem 4.4
]
First we prove that the conditions {\rm i)} and {\rm ii)} are necessary for partial exchangeability. 
Partial exchangeability implies that, for every  $n\geq 2$, $k\leq M$ finite, $(A_{m,j})_{1\leq m\leq n+1,1\leq j\leq k}$ and permutation ${\sigma}_j$ of $\{1,\dots,n\}$ ($j=1,\dots,k$)
\begin{align*}
&\mathbb P(\cap_{j\leq k,m\leq n+1}(X_{m,j}\in A_{m,j}))\\
&=\!\mathbb P(\cap_{j\leq k,m\leq n}(X_{m,j}\!\in \!A_{{\sigma}_j^{-1}(m),j})\!\cap\!\cap_{j\leq k}(\!X_{n+1,j}\!\in\! A_{n+1,j}\!)\!).
\end{align*}
Hence,
\begin{align*}
\int_{\times_{m\leq n,j\leq k}A_{m,j} }\!\!\!\!\!\!\!\!\!\!\!\!\!\!\!\!\!\!\!\!\!\!\!\!\!\!\!\!\!\!
\PP(\cap_{j=1}^k (X_{n+1,j}\in \!A_{n+1,j})\mid (x_{m,j})_{m\leq n,j\leq k})
 \PP(\cap_{m\leq n,j\leq k}(X_{m,j}\in dx_{m,j}))
\end{align*}
\begin{align*}
=\int_{\times_{m\leq n,j\leq k} A_{{\sigma}_j^{-1}(m),j} }\!\!\!\!\!\!\!\!\!\!\!\!\!\!\!\!\!\!\!\!\!\!\!\!\!\!
\!\!\!\!\!\!\!\!\!\!\!\!\!\!\!\!\!\!\PP(\cap_{j=1}^k (X_{n+1,j}\in A_{n+1,j})\mid (x_{m,j})_{m\leq n,j\leq k})
\mathbb P(\cap_{m\leq n,j\leq k} (X_{m,j}\in dx_{m,j}))
\end{align*}
\begin{align*}
=\int_{\times_{m\leq n,j\leq k} A_{m,j} }\!\!\!\!\!\!\!\!\!\!\!\!\!\!\!\!\!\!\!\!\!\!\!\!\!\!\!\!\!\!\!\!\!
\PP(\cap_{j=1}^k(X_{n+1,j}\in A_{n+1,j})\mid (x_{{\sigma}_j(m),j})_{m\leq n,j\leq k})
 \mathbb P(\cap_{m\leq n,j\leq k} (X_{m,j}\in dx_{{\sigma}_j(m),j}))
\end{align*}
\begin{align*}
=\int_{\times_{m\leq n,j\leq k } A_{m,j} }\!\!\!\!\!\!\!\!\!\!\!\!\!\!\!\!\!\!\!\!\!\!\!\!\!\!\!\!\!\!\!\!
\PP(\cap_{j=1}^k (X_{n+1,j}\in A_{n+1,j})\mid (x_{{\sigma}_j(m),j})_{j\leq k,m\leq n})
\mathbb P(\cap_{m\leq n,j\leq k} (X_{m,j}\in dx_{m,j})).
\end{align*}
Thus, $\PP(\cap_{j=1}^k (X_{n+1,j}\in A_{n+1,j})\mid (x_{m,j})_{m\leq n,j\leq k})$ $= \PP(\cap_{j=1}^k (X_{n+1,j}\in A_{n+1,j})\mid (x_{{\sigma}_j(m),j})_{j\leq k,m\leq n})$, which proves {\rm i)}. 

Let us now prove {\rm ii)}.  
By partial exchangeability, we can write, for every $n$ and $A_{m,j}$, $A_j$ and $B_j$ ($j=1,\dots,k,m=1,\dots,n$), and every $i=1,\dots,k$,
\begin{align*}
&\mathbb P\Bigl(\cap_{m\leq n,j\leq k}(X_{m,j}\in A_{m,j})\cap \cap_{j=1}^k(X_{n+1,j}\in A_j)
 \cap\cap_{j=1}^k (X_{n+2,j}\in B_j\Bigr)\\
 &=\mathbb P\Bigl(\cap_{m\leq n,j\leq k }(X_{m,j}\in A_{m,j})\cap \cap_{j\leq k,j\neq i}(X_{n+1,j}\in A_j)\\
&\qquad\cap\cap_{j\leq k, j\neq i} (X_{n+2,j}\in B_j)
\cap(X_{n+1,i}\in B_i) \cap(X_{n+1,i}\in A_i) \Bigr).
\end{align*}
Hence,
\begin{multline*}
\int_{\times_{m\leq n,j\leq k} 
 A_{m,j}\times \times_{j\leq k,j\neq i}A_j\times B_j}  \!\!\!\!\!\!\!\!\!\!\!\!\!\!\!\!\!\!\!\!\!\!\!\!\!\!\!\!\!\!\!\!\!\!\!\!\!\!\!\!\!\!\!\!\!\!\!\!\!\!\mathbb P(\cap_{m\leq n,j\leq k }(X_{m,j}\in dx_{m,j})) \\
 \int_{ A_{i}}\PP(X_{n+1,i}\in dx_{n+1,i}\mid (x_{m,j})_{m\leq n,j\leq k})\\
 \PP(X_{n+2,i}\in B_i\mid (x_{m,j})_{m\leq n+1,j\leq k})
 \end{multline*}
 \begin{multline*}
      =\int_{\times_{m\leq n,j\leq k} 
 A_{m,j}\times \times_{j\leq k,j\neq i}A_j\times B_j} \!\!\!\!\!\!\!\!\!\!\!\!\!\!\!\!\!\!\!\!\!\!\!\!\!\!\!\!\!\!\!\!\!\!\!\!\!\!\!\!\!\!\!\!\!\!\!\!\!\! \mathbb P(\cap_{m\leq n,j\leq k }(X_{m,j}\in dx_{m,j})) \\
 \int_{ B_{i}}\PP(X_{n+1,i}\in dx_{n+1,i}\mid (x_{m,j})_{m\leq n,j\leq k})\\
\PP(X_{n+2,i}\in A_i\mid (x_{m,j})_{m\leq n+1,j\leq k}).
                    \end{multline*}
 Since the equality holds for every $(A_{m,j})_{1\leq m\leq n,1\leq j\leq k}$, it follows that  {\rm ii)} holds.          

We now prove that the conditions {\rm i)}-{\rm ii)} are sufficient for partial exchangeability. Since permutations can be obtained as combinations of permutations of adjacent elements, then it is sufficient to show that for every $i=1,\dots,k$, $n\geq 2$, $s=1,\dots,n-1$, and $(A_{m,j})_{1\leq m\leq n,1\leq j\leq k}$
$$
    \mathbb P(\cap_{j=1}^k\cap_{m=1}^n(X_{m,j}\in A_{m,j}))
    =\mathbb P(\cap_{j=1}^k\cap_{m=1}^n(X_{m,j}\in A_{{\sigma}_j(m),j})),
$$
where ${\sigma}_j$ is the identity for $j\neq i$ and ${\sigma}_i$ interchanges $s$ and $s+1$. For $s=n-1$ the property is a direct consequence of {\rm ii)}. 
For $s<n-1$, we have 
\begin{multline*}
   \PP(\cap_{m\leq n,j\leq k }(X_{m,j}\in A_{{\sigma}_j(m),j}))\\
    =\int_{\times_{m\leq s-1,j\leq k }A_{m,j}}\!\!\!\!\!\!\!\!\!\!\!\!\!\!\!\!\!\!\!\!\!\!\!\!\!\!\!\!\!\!\!\!\mathbb P(\cap_{m\leq s-1,j\leq k }(X_{m,j}\in dx_{m,j}))
     \int_{\times_{j\leq k}A_{{\sigma}_j(s),j}}\!\!\!\!\!\!\!\!\!\!\!\!\!\!\!\!\!\!\!\!\!\!\!\!\!\!\!\!\PP(\cap_{j=1}^k( X_{s,j}\in dx_{s,j})\mid (x_{m,j})_{m<s,j\leq k})\\
         \int_{\times_{j\leq k} A_{{\sigma}_j(s+1),j}}\!\!\!\!\!\!\!\!\!\!\!\!\!\!\!\!\!\!\!\!\!\!\!\!\!\!\!\!\!\!\!\!\!\!\PP(\cap_{j=1}^k(X_{s+1,j}\in  dx_{s+1,j})\mid (x_{m,j})_{m\leq s,j\leq k})\\
          \int_{\times_{ s+2\leq m\leq n,j\leq k }A_{m,j}}
          \!\!\!\!\!\!\!\!\!\!\!\!\!\!\!\!\!\!\!\!\!\!\!
           \!\!\!\!\!\!\!\!\!\!\!\!\!\!\!\!\!\!
   \textstyle{\prod_{m=s+1}^{n-1}} \PP(\cap_{j=1}^k (X_{m+1,j}\in  dx_{m+1,j})\mid (x_{l,j})_{l\leq m,j\leq k})
   \end{multline*}
   \begin{multline*}
         =\int_{\times_{m\leq s-1,j\leq k}A_{m,j}}\!\!\!\!\!\!\!\!\!\!\!\!\!\!\!\!\!\!\!\!\!\!\!\!\!\!\!\!
       \mathbb P(\cap_{m\leq s-1,j\leq k }(X_{m,j}\in dx_{m,j}))
     \int_{\times_{i\leq k}A_{s,j}}\!\!\!\!\!\!\!\!\!\!\!\!\!\!\PP(\cap_{j=1}^k(X_{s,j}\in dx_{s,j})\mid (x_{m,j})_{m<s,j\leq k})\\
      \int_{\times_{j\leq k} A_{s+1,j}}\!\!\!\!\!\!\!\!\!\!\!\!\!\!\!\!\!\!\!\!\!\!\!\!\!\!
      \PP(\cap_{j=1}^k(X_{s+1,j}\in dx_{s+1,j})\mid (x_{m,j})_{m\leq s,j\leq k})\\
       \int_{\times_{s+2\leq m\leq n,j\leq k}A_{m,j}}\!\!\!\!\!\!\!\!\!\!\!\!\!\!\!\!\!\!\!\!\!\!\!\!\!\!\!\!\!\!\!\!\!\!\!\!\!
    \textstyle{\prod_{m=s+1}^{n-1}} \PP(\cap_{j=1}^k(X_{m+1,j}\in dx_{m+1,j})\mid (x_{l,j})_{l\leq m,j\leq k})
    \end{multline*}
    \begin{align*}
    =\PP(\cap_{m\leq n ,j\leq k }(X_{m,j}\in A_{m,j})),
         \end{align*}
     where the second equality comes from {\rm ii)} and the symmetry of the predictive distributions with respect to past observations.  This concludes the proof.
 \end{proof}   

\section{Proofs for Section 5 
}\label{sec:proof5}

\begin{proof}[Proof of Proposition 5.1
] $\,$
\\
{\em {\rm i)}} 
By assumption, $P_X$ has bounded support, that we assume, with no loss of generality, to be included in $[0,1]^d$.
Let $u$ be a fixed $d$-dimensional vector satisfying $\Vert u\Vert=1$. It is immediate to prove by induction that, for each $n\geq 0$, $u^T\beta_n$ is square integrable. 
Now, we prove that
$(u^T\beta_n)_{n\geq 0}$ is a martingale with respect to the filtration $(\mathcal G_n)_{n\geq 0}$, with $\mathcal G_0=\sigma(\beta_0)$, and $\mathcal G_n=\sigma(\beta_0,X_1,Y_1,\dots,X_n,Y_n)$ for $n\geq 1$. Since $E(Y_{n}-g(X_{n},\beta_{n-1}))\mid \mathcal G_{n-1},X_n)=0$, then
\begin{align*}
&E(u^T\beta_n-u^T\beta_{n-1}\mid \mathcal G_{n-1})\\&=E(E(u^T\beta_n-u^T\beta_{n-1}\mid \mathcal G_{n-1},X_{n})\mathcal G_{n-1})\\
&=\frac{1}{n\log 2}E(u^TX_nE(Y_{n}-g(X_{n},\beta_{n-1}))\mid \mathcal G_{n-1},X_n)\mid \mathcal G_{n-1})\\
&=0.
\end{align*}
Next, we show that $(u^T\beta_n)_{n\geq 0}$ is uniformly integrable. We can write that
\begin{align*}
&E((u^T\beta_n)^2)\\
&=E(u^T(\beta_{n-1}+\frac{1}{n\log 2}u^TX_n (Y_{n}-g(X_{n},\beta_{n-1})))^2)\\
&=E((u^T\beta_{n-1})^2)\\
&+\frac{1}{(n\log 2)^2}
E((u^TX_n)^2(Y_{n}-g(X_{n},\beta_{n-1}))^2),
\end{align*}
with $Y_n$, $X_n$, $u$ and $g(X_n,\beta_{n-1})$ bounded. Thus,
$$
\sup_n E((u^T\beta_n)^2)<+\infty.$$
It follows that $(u^T\beta_n)_{n\geq 0}$ is a uniformly integrable martingale,  therefore it converges almost surely and in $L^1$ to a limit random variable. Since the limit is linear in $u$, then there exists $\tilde\beta$ such that $\beta_n$ converges to $\tilde\beta$ almost surely  and in $L^1$. Moreover $E(\tilde \beta\mid\mathcal G_n)=\beta_n$. 
\\
{\em {\rm ii)}} Let  $A\in\mathcal B(\mathbb R^d)$ be such that $P_X(\partial A)=0$ almost surely. By Theorem 2 in \cite{blackwell1962},
\begin{multline*}
\PP(X_{n+1}\in A,Y_{n+1}=1\mid \mathcal G_{n})\\
 =E(\mathbf 1_A(X_{n+1})
\mathbb P(Y_{n+1}=1\mid 
\mathcal G_n,X_{n+1})\mid \mathcal G_{n})
\end{multline*}
\begin{multline*}
=E(\mathbf 1_A(X_{n+1})g(X_{n+1},\beta_n)\mid \mathcal G_{n})\\
\stackrel{a.s.}{\rightarrow}\int_Ag(x,\tilde \beta) P_X(dx).
\end{multline*}
By Lemma 8.2 in \cite{aldous1985}, the sequence $(X_n,Y_n)_{n\geq 1}$ is 
$\tilde{P}_{X,Y}$-asymptotically exchangeable, with
\begin{align*}
 &\tilde P_{X,Y}(dx,dy)\\
 &=(g(x,\tilde\beta)\delta_1(dy)+(1-g(x,\tilde\beta))\delta_{0}(dy)) \, P_X(dx).
\end{align*}
\end{proof}

\begin{proof}[Proof of Proposition 5.2
]
 The proof is based on Theorem 2.2 in \cite{crimaldi2009} and
 Lemma 4.1 in \cite{crimaldi2016} (see  \cref{th:crimaldi} and 
 \cref{lemma:series}).
 
 Let $u$ be a fixed $d$-dimensional  vector satisfying 
 $\Vert u\Vert =1$. 
 By the proof of Proposition 5.1, 
 the sequence $(u^T\beta_n)_{n\geq 0}$ is a 
 martingale with respect to the filtration $(\mathcal G_n)_{n\geq 0}$, with $\mathcal G_0=\sigma(\beta_0)$ and $\mathcal G_n=\sigma(\beta_0,X_1,Y_1,\dots,X_n,Y_n)$ for $n\geq 1$. Moreover, $(u^T\beta_n)_{n\geq 0}$ converges $\PP$-a.s. and in $L^1$ to $u^T\tilde\beta$. We apply \cref{th:crimaldi} with $M_n=u^T\beta_n$ and $M=u^T\tilde \beta$. To prove 
 that the assumption  {\rm i)} of \cref{th:crimaldi} holds, we can 
 write 
\begin{align*}
&\sup_n\sqrt n|u^T\beta_n-u^T\beta_{n-1}|\\
    &\leq \sup_n  \frac{|(Y_n-g(X_n,\beta_{n-1}))u^TX_n|}{\sqrt n\log 2},
    \end{align*}
which is integrable, since the 
random variables 
$(Y_n-g(X_n,\beta_{n-1}))u^TX_n$ are uniformly bounded. To prove that condition {\rm ii)} of \cref{th:crimaldi} holds, we  apply \cref{lemma:series} {\rm ii)} with $a_n=1$ and $b_n=n$. Let $Z_n=n^2(u^T\beta_n-u^T\beta_{n-1})^2$ for $n\geq 1$. 
Then, 
\begin{align*}
E(Z_n\mid \mathcal G_{n})
    &=(\log 2)^{-2}E( (Y_{n+1}-g(X_{n+1},\beta_{n}))^2(u^TX_{n+1})^2\mid \mathcal G_{n})\\
    & =(\log 2)^{-2} E((u^TX_{n+1})^2
    E( (Y_{n+1}-g(X_{n+1},\beta_{n}))^2\mid \mathcal G_{n},X_{n+1})
    \mid \mathcal G_{n},)\\
    &=(\log 2)^{-2}E((u^TX_{n+1})^2
    g(X_{n+1},\beta_{n})(1-g(X_{n+1},\beta_{n})))\mid \mathcal G_{n})\\
    & \stackrel{a.s.}{\rightarrow} (\log 2)^{-2}\int g(x,\tilde \beta)(1-g(x,\tilde \beta)(u^Tx)^2 P_X(dx)\\
    & =(\log 2)^{-2}u^T \int x x^T (g(x,\tilde \beta)(1-g(x,\tilde \beta))P_X(dx)\; u.
\end{align*}
By \cref{lemma:series} {\rm ii)}, 
$$
n\sum_{k\geq n}(u^T\beta_k-u^T\beta_{k-1})^2 \stackrel{a.s.}{\rightarrow}u^TUu,
$$
with
$$
    U=(\log 2)^{-2}\int x x^Tg(x,\tilde \beta)(1-g(x,\tilde \beta) P_X(dx).
$$
It follows from \cref{th:crimaldi} that
$$
E\left(e^{it\sqrt n(u^T\tilde \beta-u^T\beta_{n+1})}\mid \mathcal G_{n}\right)\stackrel{a.s.}{\rightarrow}e^{-\frac 1 2 t^2u^TU u}.
$$
Since this is true for every unitary vector $u$, then, for every $s$,
$$
E\left(e^{is^T\sqrt n(\tilde \beta-\beta_{n+1})}\mid \mathcal G_{n}\right)\stackrel{a.s.}{\rightarrow}\exp({-\frac 1 2 s^TUs}),
$$
which concludes the proof.
\end{proof}

\begin{proof}[Proof of Proposition 5.3
]$\;$\\
{\em {\rm i)}} Let $u$ be a $d$-dimensional vector satisfying $\Vert u\Vert =1$. 
We  prove that $u^TV_nu$ converges to $u^TUu$, $\PP$-a.s.
We can write that 
\begin{align*}
&\frac 1 n \sum_{k=1}^n k^2u^T(\beta_k-\beta_{k-1})(\beta_k-\beta_{k-1})^Tu
    =\frac 1 n \sum_{k=1}^n k^2(u^T\beta_k-u^T\beta_{k-1})^2\\
    &=\frac 1 n \sum_{k=1}^n k^2E((u^T\beta_k-u^T\beta_{k-1})^2\mid \mathcal G_{k-1})+\frac 1 n \sum_{k=1}^n k^2\{(u^T\beta_k-u^T\beta_{k-1})^2\\
    &\quad -E((u^T\beta_k-u^T\beta_{k-1})^2\mid \mathcal G_{k-1})\}.
\end{align*}
For the first term, we can write that
\begin{align*}
&\frac 1 n \sum_{k=1}^n k^2E((u^T\beta_k-u^T\beta_{k-1})^2\mid \mathcal G_{k-1})\\&
=(\log 2)^{-2}\frac 1 n \sum_{k=1}^n E((Y_k-g(X_k,\beta_{k-1}))^2 (u^TX_k)^2\mid \mathcal G_{k-1}).
\end{align*}
From the proof of Proposition 5.2, 
we deduce that
\begin{align*}
    &E((Y_k-g(X_k,\beta_{k-1})^2(u^TX_k)^2\mid \mathcal G_{k-1})\\
    &\stackrel{a.s.}{\rightarrow }
\int g(x,\tilde \beta)(1-g(x,\tilde \beta)
    (u^T x)^2 P_X(dx),
    \end{align*}
as $k\rightarrow\infty$.
Hence, by \cref{lemma:series} {\rm i)} with $a_n=1$ and $b_n=n$,
\begin{align*}
    &\frac 1 n \sum_{k=1}^n k^2E((u^T\beta_k-u^T\beta_{k-1})^2\mid \mathcal G_{k-1})\\
&\stackrel{a.s.}{\rightarrow }(\log 2)^{-2}\int (g(x,\tilde \beta)(1-g(x,\tilde \beta))
    (u^T x)^2 P_X(dx)\\
    &=u^TUu.
\end{align*}
Let us now prove that
\begin{align*}
    &\frac 1 n \sum_{k=1}^n k^2\Bigl((u^T\beta_k-u^T\beta_{k-1})^2
 -E((u^T\beta_k-u^T\beta_{k-1})^2\mid \mathcal G_{k-1})\Bigr)\stackrel{a.s.}{\rightarrow}0.
\end{align*}
To this aim we can invoke the martingale law of large numbers (\cite{hall1980}, Theorem 2.18), that requests to show that
\begin{align*}
&\sum_{k=1}^\infty \frac{1}{k^2}E\Bigl(
k^2\bigl((u^T\beta_k-u^T\beta_{k-1})^2 -E((u^T\beta_k-u^T\beta_{k-1})^2)\bigr)^2\mid \mathcal G_{k-1}\Bigr)<+\infty.
\end{align*}
We can write that
\begin{align*}
&\sum_{k=1}^\infty \frac{1}{k^2}E\Bigl(
k^2\bigl((u^T\beta_k-u^T\beta_{k-1})^2
-E((u^T\beta_k-u^T\beta_{k-1})^2)\bigr)^2\mid \mathcal G_{k-1}\Bigr)\\
&\quad\leq \sum_{k=1}^\infty k^2E\Bigl((u^T\beta_k-u^T\beta_{k-1})^4\mid \mathcal G_{k-1}\Bigr)\\
&\quad\leq (\log 2)^{-4}\sum_{k=1}^\infty \frac{1}{k^2}E\Bigl((Y_k-g(X_k,\beta_{k-1}))^4
(u^TX_k)^4\mid \mathcal G_{k-1}\Bigr)<+\infty,
\end{align*}
where the last inequality holds true since the terms  $(Y_k-g(X_k,\beta_{k-1})u^TX_k)$ are uniformly bounded. We can conclude that $u^TV_nu$ converges almost surely to $u^TUu$. Since this is true for every unitary vector $u$, then $V_n$ converges almost surely to $U$, as $n\rightarrow\infty$. \\
{\em {\rm ii)}} From Proposition 5.2, 
we know that, for $\omega$ in a set with probability one,
\begin{align*}
&E(\exp({is^T\sqrt n(\tilde\beta-\beta_n)})\mid \mathcal G_{n})(\omega)
\rightarrow \exp({-\frac 1 2 s^TU(\omega )s}),
\end{align*}
and the convergence is uniform with respect to $s$ on compact sets.
If $\omega$ is such that  $V_n(\omega)$ converges to $U(\omega)$, then
\begin{align*}
E(\exp({i(s^TV_n^{-1/2})\sqrt n(\tilde\beta-\beta_n)})\mid \mathcal G_{n})(\omega)
&\rightarrow \exp\left(-\frac 1 2 s^TU(\omega)^{-1/2}U(\omega) U^{-1/2}(\omega )s\right)\\
&=\exp({-\frac 1 2 s^Ts}).
\end{align*}
This concludes the proof.
\end{proof}

\bibliographystyle{imsart-number} 
\bibliography{statsc}